\documentclass{amsart}

\pdfoutput=1
\usepackage{amsthm,amsmath,amssymb}

\usepackage[utf8]{inputenc}

\usepackage{enumerate}
\usepackage{esint}
\usepackage{tikz}
\usepackage{cite}
\usepackage{hyperref}

\usepackage{xcolor}

\hypersetup{
    colorlinks,
    linkcolor={red!50!black},
    citecolor={blue!50!black},
    urlcolor={blue!80!black}
}

\theoremstyle{plain}
  \newtheorem{theorem}{\bf Theorem}[section]
  \newtheorem{proposition}[theorem]{\bf Proposition}
  \newtheorem{lemma}[theorem]{\bf Lemma}

\theoremstyle{remark}
  \newtheorem{remark}[theorem]{\bf Remark}

  \numberwithin{equation}{section}
  
\newcommand{\RR}{\mathbb{R}} %Real numbers%
\newcommand{\SP}{\mathbb{S}} %sphere%

\newcommand{\norm}[1]{\lVert#1\rVert} %norma%

\newcommand{\supp}{\mathrm{supp \hspace{.1 mm}}} %support%
\newcommand{\jp}[1]{<#1>} %parentesis japonés%

\newcommand{\pv}{{\it P.V.}} %principal value%

\newcommand{\ds}[1]{d\sigma_{#1k}} %medida EWALD%

\begin{document}

%%%%%%%%%%%%%%%%%%%%%%%%%%%%%%%%%%%%%%%%%%%%%%%%%%%%%%%%%%%%%%%%%%%%%%%%%%%

\title[Recovery of the singularities of a potential]{Fixed angle scattering: Recovery of  singularities and its limitations }
\author{Cristóbal J.  Meroño}
\thanks{Cristóbal J.  Meroño, Departamento de Matemáticas, Universidad Autónoma de Madrid. E-mail: cristobal.meronno@uam.es}
\date{\today}

%%%%%%%%%%%%%%%%%%%%%%%%%%%%%%%%%%%%%%%%%%%%%%%%%%%%%%%%%%%%%%%%%%%%%%%%%%%

\begin{abstract}
We prove that in dimension $n \ge 2$ the main singularities of a complex potential $q$ having a certain a priori regularity are contained  in the Born approximation $q_\theta$ constructed from fixed angle scattering data. Moreover, ${q-q_\theta}$ can be up to one derivative more regular than $q$ in the Sobolev scale. In fact, this result is optimal, we construct a family of compactly supported and radial potentials for which it is not possible  to have more than one derivative gain. Also, these functions show that for $n>3$, the maximum derivative gain can be very small for potentials  in the Sobolev scale not having a certain a priori level of regularity which grows with the dimension.
\end{abstract}

%%%%%%%%%%%%%%%%%%%%%%%%%%%%%%%%%%%%%%%%%%%%%%%%%%%%%%%%%%%%%%%%%%%%%%%%%%%

\maketitle

%%%%%%%%%%%%%%%%%%%%%%%%%%%%%%%%%%%%%%%%%%%%%%%%%%%%%%%%%%%%%%%%%%%%%%%%%%%

\section{Introduction and main results}

The central problem in inverse scattering for the Schrödinger equation is to recover a potential $q(x)$, $x \in \RR^n$, from the scattering data, the so-called scattering amplitude $u_\infty$. The scattering amplitude measures the far field response of the Hamiltonian $H:=-\Delta + q$ to incident plane waves. 
In applications to physics one of the  most used methods  in scattering is to construct from the data the Born approximation of the potential, essentially a linear approximation to the inverse problem. There  are several natural ways to construct a Born approximation, and they differ on how the scattering data are used. The main examples are related to the backscattering problem, the fixed angle scattering problem and  the full data scattering problem, which we will introduce in detail in the next section.  

 As the name suggests, the Born approximation  is a good approximation for potentials satisfying certain smallness conditions. But, motivated by applications, an interesting question from a mathematical point of view  is to establish how much and what kind of information does it contain about a potential $q(x)$ which does not necessarily satisfy any smallness condition.  In \cite{PSo} the full data Born approximation  was introduced and they showed that in  dimension three it must contain the leading singularities of $q(x)$. Since then, this approach has received a great amount of attention in all the three different scattering problems mentioned. Due to its  radial symmetry properties the most studied cases are the backscattering problem (see, among others \cite{OPS,RV,BM09,Re,RRe,BFRV13} and \cite{GU,Esw} for a different approach) and the full data scattering problem (see \cite{PSo,PSe,PSS} for real potentials and \cite{BFRV10} for complex potentials). In the case of  fixed angle scattering we mention \cite{R} for results in dimension $n \ge 2$, \cite{ser} in $n=2$, and \cite{BFPRM1} where the techniques introduced in \cite{R} are applied to fixed angle scattering in elasticity.

Usually the singularities are measured in the Sobolev, Hölder or integrability scales. In this work we are going to use the Sobolev scale with integrability exponent $p=2$, see \cite{R,BFRV10} for results in more general Sobolev spaces. If $\jp{x} = (1+|x|^2)^{1/2}$ and $\alpha \in \RR$, we introduce  the (Bessel) fractional derivative operator $\jp{D}^{\alpha}$ given by the Fourier symbol $\jp{\xi}^\alpha$, and the weighted Sobolev spaces
\[W_\delta^{\alpha,p}(\RR^n) := \{ f\in \mathcal{S}' : \norm{\jp{\cdot}^\delta \jp{D}^{\alpha} f}_{L^p(\RR^n)} <\infty\}, \]
for $\delta \in \RR$. We say that $f\in W^{\alpha,p}_{loc}(\RR^n)$ if $\phi f\in W^{\alpha,p}(\RR^n)$ for every $\phi \in C^\infty_c(\RR^n)$, also we use the notation $L_\delta^p(\RR^n) := W^{0,p}_\delta(\RR^n)$.

 As we shall see in the next section, the Born approximations $q_\theta$ of fixed angle scattering and the Born approximation $q_F$ of full data scattering are related to the potential through the Born series expansion, respectively
\begin{equation*} 
 {q}_\theta  \sim  {q} + \sum_{j=2}^{\infty}{Q_{\theta,j}(q)}, \hspace{3mm} \text{and} \hspace{3mm}  {q}_F  \sim  {q} + \sum_{j=2}^{\infty}{Q_{F,j}(q)},
 \end{equation*}
where $Q_{\theta,j}(q)$ and $Q_{F,j}(q)$, are certain multilinear operators  describing the multiple dispersion of waves (we use the $\sim$ symbol to avoid claiming anything about convergence yet).  We will call the $Q_{\theta,2}$ operator the double dispersion operator of fixed angle scattering (and analogously in full data scattering). The $\theta$ subindex in the notation is due to the fact that in fixed angle scattering the plane waves have a fixed direction of propagation $\theta \in \SP^{n-1}$. We can  now introduce the main  theorems in this work
\begin{theorem} \label{teo.1}
Let $n\ge 2$ and $\beta \ge 0$. Assume that at least one of the statements $q-q_{\theta} \in W_{loc}^{\alpha,2}(\RR^n)$ or $q-q_{F} \in W_{loc}^{\alpha,2}(\RR^n)$  holds for every $q\in W^{\beta,2}(\RR^n)$ compactly supported, radial, and real. Then $\alpha$ necessarily satisfies,
\begin{equation*}
\alpha \le  \begin{cases}
 \,\, 2\beta - (n-4)/2, \hspace{4 mm} if \hspace{16mm} m \le \beta < (n-2)/2,\\
 \, \, \beta + 1, \hspace{18.5mm} if \hspace{4mm} (n-2)/2 \le \beta<\infty ,
         \end{cases} 
\end{equation*}
\begin{equation} \label{eq.m}
where \quad  m = (n-4)/2 + 2/(n+1).
\end{equation}
\end{theorem}
As far as we know, this is the first time that  necessary conditions  are given for the regularity of $q-q_\theta$ and $q-q_F$.  As a consequence, it has been established that in general it is not possible to have more than one derivative gain in the Sobolev scale. In fact, in the worst case $\beta=m$ we have $(\alpha -\beta) \le 2/(n+1)$ which goes to zero as $n$ grows (see Figure \ref{fig.dim3yn}). This is true even if we consider only radial, real and compactly supported potentials. We remark that for $n \ge 4$ and $0 \le \beta < m$, it is not even known if the (high frequency) Born series converges in both  scattering problems (see Proposition \ref{prop.Qj} below). 
We have also the following positive results.
\begin{theorem}[Recovery of singularities] \label{teo.recovery} 
Assume that $q \in W^{\beta,2}(\RR^n)$  is a  compactly supported function. Then $q-q_{\theta} \in W^{\alpha,2}(\RR^n)$,  modulo a $C^\infty$ function, if $0 \le \beta <\infty$ and the following conditions hold,
\begin{equation*}  
 \alpha < \begin{cases}
 \,\, 2\beta - (n-3)/2, \hspace{4 mm} if \hspace{4mm} (n-3)/2 <\beta < (n-1)/2,\\
 \, \, \beta + 1, \hspace{18.5mm} if \hspace{4mm} (n-1)/2 \le \beta<\infty .
         \end{cases} 
\end{equation*} 
\end{theorem}
The condition $q-q_{\theta} \in W^{\alpha,2}(\RR^n)$,  modulo a $C^\infty$ function implies that $q-q_{\theta} \in W_{loc}^{\alpha,2}(\RR^n)$. This theorem shows that there is a $1^-$ derivative gain when $\beta\ge(n-1)/2$, which, except for the limiting case $\alpha=\beta+1$, is the best possible result by Theorem $\ref{teo.1}$. It also improves the results of \cite{R} in the spaces $W^{\alpha,2}(\RR^n)$ for every value of $\beta$ (see Figure \ref{fig.dim3yn}).

\begin{figure} \label{fig.dim3yn}
\begin{tikzpicture}[scale=1.5]
 \draw[<->] (3,0) node[below]{$\beta$} -- (0,0) --
(0,2.3) node[left]{$\alpha-\beta$};
  \draw [dotted,gray, very thin] (0.5,1) -- (0.5,0);
   \draw [dotted, gray, very thin] (0,1) -- (.5,1);
    \draw [dotted, gray, very thin] (1,1) -- (1,0);
    \draw [dotted, gray, very thin] (0,.5) -- (1.5,.5);
     \draw [dotted,gray, very thin] (1.5,0.5) -- (1.5,0);
    \draw  [dashed, semithick,red] (0,0.5) -- (0.5,1);
 \draw  [dashed, semithick,red]  (0.5,1) -- (3,1);
   \draw [dash dot ] (0,0) -- (1.5,.5);
    \draw  [semithick,blue](0,0) -- (1,1);
 \draw  [semithick,blue] (1,1) -- (3,1);
  \node at (1,2) {$n=3$};
 \node [left] at (0,1) { \tiny $1$ };
 \node [left] at (0,0.5) { \tiny $\frac{1}{2}$ };
 \node [below] at (0.5,0) { \tiny $\frac{1}{2}$ };
  \node [below] at (1,0) { \tiny $1$ };
   \node [below] at (1.5,0) { \tiny $\frac{3}{2}$ };
\end{tikzpicture}
\begin{tikzpicture}[scale=1.5]
1. \draw[<->] (3.5,0) node[below]{$\beta$} -- (0,0) --
(0,2.3) node[left]{$\alpha-\beta$};
 \draw [dotted, gray, very thin] (0,1) -- (2,1);
 \draw [dotted, gray, very thin] (2,1) -- (2,0);
  \draw [dotted, gray, very thin] (3,.5) -- (3,0);
  \draw [dotted,gray,very thin] (2.5,1) -- (2.5,0);
   \draw [dotted, gray,very thin] (1.4,0) -- (1.4,0.4);
    \draw  [dashed, semithick,red](1,0) -- (2,1);
 \draw  [dashed, semithick,red]  (2,1) -- (3.5,1); 
    \draw [dash dot ] (1.8,0) -- (3,.5);  
 \draw  [semithick,blue] (1.5,0) -- (2.5,1);
 \draw  [semithick,blue] (2.5,1) -- (3.5,1);
  \node at (1,2) {$n>3$};
 \node [left] at (0,1) { \tiny $1$ };
  \node [ below] at (3,0) { \tiny $\frac{n}{2}$ };
 \node [below] at (2.6,0) { \tiny $\frac{n-1}{2}$ };
 \node [below] at (2.1,0) { \tiny $\frac{n-2}{2}$ };
\node [ below] at (1.6,0) { \tiny $\frac{n-3}{2}$ };
   \node [below] at (.9,0) { \tiny $\frac{n-4}{2}$ };
  \node [below] at (1.3,0) { \tiny $m$ };
\end{tikzpicture}
   \caption{The (red) dashed line represents the limitation on the regularity gain given in Theorem  $\ref{teo.1}$ for $q-q_{\theta}$ and in Theorem $\ref{teo.Q2}$ for the $Q_{\theta,2}$ operator, and the solid (blue)  line represents the positive results given in Theorem $\ref{teo.recovery}$. The dot dashed line represents the previously known positive results of \cite{R}.}
\end{figure}
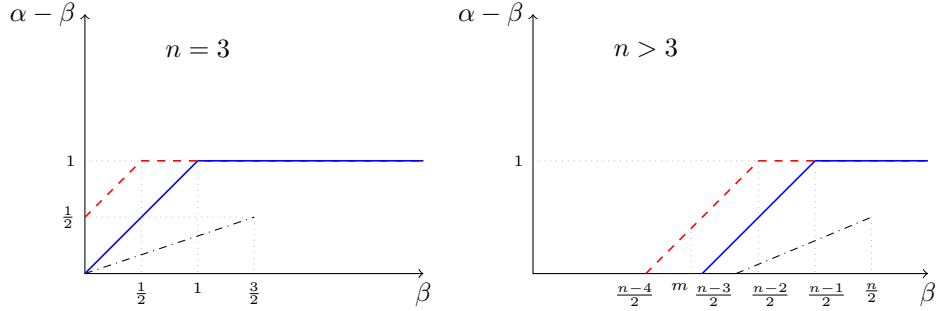

The key point to prove Theorem \ref{teo.recovery} is to obtain new estimates of the double dispersion operator $Q_{\theta,2}$ that we now introduce. Consider a constant $C_0>1$ and  let $0 \le \chi(\xi) \le 1$, $\xi\in \RR^n$, be a smooth cut-off function such that
 \begin{equation} \label{eq.cutoff1}
 {\chi}(\xi)= 1 \;\; \text{if} \;\; |\xi|>2C_0 \;\; \text{and} \;\; \chi(\xi)=0 \;\; \text{if} \;\; |\xi|<C_0.
 \end{equation}
 Then we define
\begin{equation}  \label{eq.cutoff2}
 \widehat{\widetilde{Q}_{\theta,j}(q)}(\xi) := {\chi}(\xi)\widehat{{Q}_{\theta,j}(q)}(\xi). 
 \end{equation}
Hence,  $\widetilde{Q}_{\theta,j}(q)$ and ${Q}_{\theta,j}(q)$ differ just in a $C^\infty$ function.
\begin{theorem}  \label{teo.Q2}
Let $n\ge 2$ and $q\in W_2^{\beta,2}(\RR^n)$ with $0 \le \beta < \infty$.  Then 
\begin{equation*}
\norm{\widetilde {Q}_{\theta, 2}(q)}_{{W}^{\alpha,2}}\le C  \norm{q}_{W^{\beta,2}_2}^2,
\end{equation*}
if the following conditions also hold,
\begin{equation} \label{eq:rangepos}
 \alpha < \begin{cases}
 \,\, 2\beta - (n-3)/2, \hspace{4 mm} if \hspace{5mm} (n-3)/2 <\beta < (n-1)/2,\\
 \, \, \beta + 1, \hspace{18.5mm} if \hspace{5mm} (n-1)/2 \le \beta<\infty .
         \end{cases} 
\end{equation} 
Conversely, if we assume that ${Q_{\theta,2}(q) \in {W}_{loc}^{\alpha,2}(\RR^n)}$ for all ${q\in W^{\beta,2}(\RR^n)}$ real, radial and compactly supported, then $\alpha$ necessarily satisfies
\begin{equation} \label{eq.remarkable} 
\alpha\le \min (\beta + 1, 2\beta  -(n-4)/2).
\end{equation}
\end{theorem}

The necessary condition (\ref{eq.remarkable}) also holds for the double dispersion operator $Q_{F,2}$ in full data scattering (see Theorem \ref{teo.count}).  It is interesting to notice that this implies that the estimates of $Q_{F,2}$ given in \cite[Theorem 1.2]{BFRV10} are the best possible for every $0 \le \beta < \infty$ and $p=2$ (except for the limiting case). Hence, in full data scattering the regularity gap between the positive and negative results for the double dispersion operator has been closed for the complete range of $\beta$ and $n\ge2$. Unfortunately, this is not translated into optimal results of recovery of singularities, since also the estimates of higher order $Q_{F,j}$ operators  should be improved in the range $(n-4)/2 < \beta \le (n-2)/2$ (see \cite[Theorem 1.1]{BFRV10}).  

The paper is structured as follows. In section \ref{sec.2} we introduce in detail the fixed angle and full data scattering problems, and we sketch the proof of Theorem \ref{teo.recovery} from Theorem \ref{teo.Q2}. In sections \ref{sec.3} and \ref{sec.4} we prove the estimates of the $\widetilde Q_{\theta,2}(q) $ operator given in Theorem \ref{teo.Q2}. The proof of this theorem  is finished in section \ref{sec.5}, where we construct the counterexamples from which condition (\ref{eq.remarkable}) follows and we prove Theorem \ref{teo.1}. In the appendix we give some technical results used in the rest of the paper.

\section{The fixed angle and full data scattering problems} \label{sec.2}

 We now introduce with more detail the fixed angle and full data scattering inverse problems, following  \cite{BFRV10} (see, for example, \cite[chapter V]{eskin} for a more general introduction to scattering).  Consider  the scattering solution $u_s(k,\theta,x)$, $k\in(0,\infty)$, $\theta\in \SP^{n-1}$,  of the stationary Schrödinger equation which satisfies
\begin{equation} \label{eq.int.1}
\begin{cases}
(-\Delta + q -k^2)u=0 \\
u(x) = e^{i k \theta \cdot x} + u_s(k,\theta,x) \\
\lim_{|x| \to \infty} (\frac{\partial u_s}{\partial r} - iku_s)(x) = o(|x|^{-(n-1)/2}),
\end{cases}
\end{equation}
where the last line is the outgoing Sommerfeld radiation condition necessary for uniqueness.  If $q$ is compactly supported, a solution $u_s$ of (\ref{eq.int.1}) must have the following asymptotic behavior when $|x| \to \infty$
\[u_s(k,\theta,x) =  C |x|^{-(n-1)/2}k^{(n-3)/2} e^{ik|x|} u_\infty(k,\theta,x/|x|) + o(|x|^{-(n-1)/2}) ,\]
 for a certain function $u_\infty(k,\theta,\theta')$, $k\in (0,\infty)$, $\theta,\theta' \in \SP^{n-1}$. As mentioned in the introduction, $u_\infty$ is the scattering amplitude or far field pattern, and it is given by the expression
\begin{equation} \label{eq.int.2}
u_\infty (k,\theta,\theta') = \int_{\RR^n} e^{-ik\theta'\cdot y} q(y) u(y) \, dy,
\end{equation}
where it is important to notice that $u$ depends also on $k$ and $\theta$ (see, for example, \cite{notasR} for a proof of this fact when $q\in C^\infty_c(\RR^n)$). 

Applying the outgoing resolvent of the Laplacian $R_k$ in the first line of (\ref{eq.int.1}), where 
\begin{equation} \label{eq:resolvent}
\widehat{R_k(f)}(\xi)  = (-|\xi|^2+k^2 + i0)^{-1}\widehat{f}(\xi),
\end{equation}
we obtain the Lippmann-Schwinger integral equation
\begin{equation} \label{eq.int.3}
u_s=R_k(qe^{ik\theta \cdot (\cdot)}) + R_k(qu_s(k,\theta,\cdot)).
\end{equation}
The existence and uniqueness of the scattering solution of (\ref{eq.int.1}) follows from a priori estimates for the resolvent operator $R_k$ and the previous integral equation (\ref{eq.int.3}). In the case of real potentials, this can be shown with  the help of Fredholm theory, see, for example, \cite[Chapter 6]{notasR}. Otherwise, since the norm of the operator $T(f)= R_k(qf)$  decays to zero as $k \to \infty$ in appropriate function spaces, we can also use a Neumann series expansion in (\ref{eq.int.3}) which will be convergent for  $k >k_0$ (in general $k_0 \ge 0$  will depend on some a priori bound of $q$). For our purposes it is enough to consider compactly supported $q \in L^r(\RR^n)$, $r>n/2$. Notice that by the Sobolev embedding this is satisfied if $q\in W^{\beta,2}(\RR^n)$ with $\beta>(n-4)/2$. See \cite[p. 511]{BFRV10} for more details and references.

We can now introduce the inverse scattering problems. If we insert (\ref{eq.int.3}) in (\ref{eq.int.2}),  we can expand  the Lippmann-Schwinger equation in a Neumann series, as we mentioned before.  Then we obtain the Born series expansion relating  the scattering amplitude and the Fourier transform of the potential
\begin{align}
\nonumber u_\infty (k,\theta,\theta') = \widehat{q}(\eta) + \sum^l_{j=2}\int_{\RR^n} e^{-ik \theta'\cdot y} (qR_k)^{j-1}(q(\cdot)e^{ik\theta \cdot (\cdot)} )(y) \,dy\\
 \label{eq.int.4} + \int_{\RR^n} e^{-ik \theta'\cdot y} (qR_k)^{l-1}(q(\cdot)u_s(k,\theta, \cdot) )(y) \,dy,
\end{align}
where $\eta = k(\theta'-\theta)$, $k>k_0$ and the last is the error term. 
The problem of determining $q$ from the knowledge of the scattering amplitude is formally overdetermined in the sense that  the data $u_\infty(k,\theta,\theta')$ is described by $2n-1$ variables, while the unknown potential $q(x)$ has only $n$. There are different ways to deal with the extra parameters.

{\it Backscattering problem.} In backscattering, only knowledge of $u_\infty(k,\theta,-\theta)$  is assumed  for all $k>k_0$ and $\theta \in \SP^{n-1}$.  Then the problem is formally well determined, and the Born approximation $q_{B}$ is defined by the identity,
\begin{equation*} 
 \widehat{q_{B}}(\eta):= u_\infty(k,\theta,-\theta), 
 \end{equation*}
 where $\eta= -2k\theta$ and $k>k_0$. We will not treat this problem further in this work, see the introduction for references.
 
 {\it Fixed angle scattering problem}. In the fixed angle scattering problem one assumes  knowledge of $u_\infty(k,\theta,\theta')$ only for a fixed $\theta \in \SP^{n-1}$ and the opposite unit vector $-\theta$, and for all $k>k_0$ and $\theta' \in \SP^{n-1}$. Then the problem is  formally well determined. Uniqueness for the inverse problem of recovering $q(x)$ from the previous data  is still an open question.  Generic uniqueness and uniqueness  for small potentials has been obtained in \cite{stefanov} for  potentials in dimension 3 with certain  smoothness conditions. Also, it has been shown in \cite{BLM} that if the scattering amplitude vanishes for a fixed $\theta$, then $q$ has to be zero. 
 
  Now, for a fixed $\theta$, the identity $\eta = k(\theta'-\theta)$ gives us a diffeomorphism (a chart) from $(0,\infty)\times \SP^{n-1}$ to $H_\theta \subset \RR^n$, where
\begin{equation*} 
H_\theta := \{\eta \in \RR^n: \eta \cdot \theta <0 \},
\end{equation*} 
is an open half space of $\RR^n$. Inverting this diffeomorphism, we obtain that for $\eta \in H_\theta$, the relation $\eta = k(\theta'-\theta)$ is satisfied if and only if
\begin{equation} \label{eq.keta}
 k(\eta,\theta) := - \frac{|\eta|^2}{\theta \cdot \eta} \hspace{3mm}\text{and} \hspace{3mm} \theta'(\eta,\theta) := k^{-1}(\eta + k\theta).
\end{equation}
 We notice that the condition $k(\eta,\theta)>k_0$ holds  if we ask $|\eta|>C_0$ for any constant $C_0>2k_0$ since we always have that $2k\ge |\eta|$. Therefore for $|\eta|>C_0$, we can define the Born approximation of fixed angle scattering as follows,
\begin{equation} \label{eq.born}
 \widehat{q_{\theta}}(\eta):= \begin{cases}
 u_\infty(k(\eta, \theta),\theta,\theta'(\eta, \theta)), \hspace{14mm} \text{when}  \hspace{4mm} \eta\in H_\theta,\\
u_\infty(k(\eta,-\theta),-\theta,\theta'(\eta,-\theta)), \hspace{6mm} \text{when}  \hspace{4mm}\eta\in H_{-\theta},
 \end{cases}
 \end{equation}
where we  need the angle $-\theta$ to generate the opposite half space to cover a full measure subset of $\RR^n$. For real potentials, using symmetries of the scattering data $u_\infty$, is possible to reduce the data to only one unit vector $\theta$  (see \cite{PSo}) but we consider directly the general case. 
Assuming convergence of the series, we can write  (\ref{eq.int.4}) taking $l \to \infty$  as follows,
\begin{equation} \label{eq.int.5}
 \chi(\eta)\widehat{q}_\theta(\eta)  =  \chi(\eta)\widehat{q}(\eta) + \sum_{j=2}^{\infty} \widehat{\widetilde{Q}_{\theta,j}(q)}(\eta) ,
 \end{equation}
where $\widehat{\widetilde{Q}_{\theta,j}(q)}(\eta) = {\chi}(\eta)\widehat{{Q}_{\theta,j}(q)}(\eta)$ (see (\ref{eq.cutoff2})), and
\begin{equation} \label{eq.QjB}
\widehat{Q_{\theta,j}(q)}(\eta) := {B_{\theta,j}(q)}(\eta) + {B_{-\theta,j}(q)}(\eta),
\end{equation}
 if we define, for $\theta \in \SP^{n-1}$, 
\begin{equation} \label{eq.B}
{B_{\theta,j}(q)}(\eta) := \begin{cases}
\int_{\RR^n} e^{-ik \theta'\cdot y} (qR_k)^{j-1}(qe^{ik\theta\cdot (\cdot)} )(y) \,dy, \hspace{4mm}\text{if} \hspace{3mm} \eta \in H_\theta, \\
0 \hspace{58mm}\text{if} \hspace{3mm} \eta \notin H_\theta,
\end{cases}
\end{equation}
with $k=k(\eta,\theta)$ and $\theta'= \theta'(\eta,\theta)$  given by (\ref{eq.keta}). Notice we have avoided giving a definition of $\widehat{q_\theta}(\eta)$ for all $|\eta|<C_0$. As a consequence, from now on it will be understood that the function $q_\theta(x)$ is defined {\it modulo a $C^\infty$ function} (when the potentials are real, $u_\infty$ is defined for all $k>0$ so we also have, by (\ref{eq.born}) a definition of $\widehat q_\theta(\eta)$ for low frequencies).

{\it Full data scattering}. In this case we  construct a Born approximation considering the values  $u_\infty(k,\theta,\theta')$ for all $k>k_0$ and $\theta,\theta' \in \SP^{n-1}$.  To do that, fix $\eta \in \RR^n$, and consider the Born approximation of fixed angle scattering for every $\theta \in \SP^{n-1}$. We are going to deal with the overdeterminacy by taking an average on the $\theta$ parameter.  Taking the average  of (\ref{eq.int.5}) in $\SP^{n-1}$, we obtain
\begin{equation} \label{eq.Fborn}
  \chi(\eta)\widehat{q}(\eta) =  \chi(\eta)\widehat{q}_{F}(\eta) + \sum_{j=2}^{\infty}   \widehat{\widetilde{Q}_{F,j}(q)}(\eta),
 \end{equation} 
 where $\widehat{\widetilde{Q}_{F,j}(q)}(\eta):={\chi}(\eta) \widehat{{Q}_{F,j}(q)}(\eta)$ and $q_F$ is the Born approximation of full data scattering,
 \[ \widehat{q_{F}}(\eta):= \fint_{\SP^{n-1}} \widehat{q_\theta}(\eta) \, d\sigma(\theta), \hspace{3mm} \text{and} \hspace{3mm} \widehat{Q_{F,j}(q)}(\eta) = \fint_{\SP^{n-1}} \widehat{Q_{\theta,j}(q)}(\eta) \, d\sigma(\theta).\]
 In this work we are not going to need any explicit formula for the $Q_{F,j}$ operators, only to observe that  by  (\ref{eq.QjB}) we have
\begin{equation} \label{eq.BjF}
\widehat{Q_{F,j}(q)}(\eta) =  \frac{2}{|\SP^{n-1}|} \int_{\{\theta\in \SP^{n-1}: \, \eta \cdot\theta <0\}} {B_{\theta,j}(q)}(\eta) \, d\sigma(\theta),
\end{equation}
 (see \cite{BFRV10} for a specific computation of the  Fourier symbol of  $Q_{F,2}$).
 
We now introduce the main result giving convergence of the series (\ref{eq.int.5}) in Sobolev spaces.
\begin{proposition}     \label{prop.Qj}
Let  $n\ge 2$, $j\ge 2$ and $\max(0,m) \le \beta<\infty$, where $m$ was defined in $(\ref{eq.m})$. Then if $q\in W^{\beta,2}(\RR^n)$  is compactly supported, $\widetilde Q_{\theta,j}(q) \in W^{\alpha,2}(\RR^n)$ 
 if $\alpha<\alpha_j$, with
\begin{equation} \label{eq.alphaj}
\alpha_j = \beta -\frac{1}{2} + (j-1)  - (j-1) \frac{(n-1)}{2}  \max{\left( 0,\frac{1}{2} - \frac{\beta}{n} \right )}  .
\end{equation}
Moreover,  for every $\alpha >0$, there exists an $l\ge 2$  such that the series $ \sum_{j=l}^{\infty} \widetilde Q_{\theta,j}(q)$
converges absolutely in $W^{\alpha,2}(\RR^n)$ provided we take ${C_0=C\norm{q}^{1/\varepsilon}_{W^{\beta,2}}}$ in $(\ref{eq.cutoff1})$ and $(\ref{eq.cutoff2})$ for a large constant $C=C(n,\alpha,\beta,\supp \, q)$ and a certain $\varepsilon = \varepsilon(n,\beta)>0$.
\end{proposition}
In dimension $n>4$ this result gives  convergence of the tail of the high frequency Born series provided $q \in W^{\beta,2}(\RR^n)$ with $\beta>m>(n-4)/2$. This means that $q$ must have more a priori regularity as $n$ increases. In part, this is to be expected, since the condition $q \in L^r$, $r>n/2$ is necessary for the existence of scattering solutions, as mentioned previously. This also suggests that for high dimension it could be more adapted to the problem at hand to measure  the regularity of $q$ (and the regularity gain) in the space $W^{\beta,r}(\RR^n)$, for some $r\ge n/2$ and instead of in $W^{\beta,2}(\RR^n)$, but this is a more difficult problem. 

Proposition \ref{prop.Qj} is originally from \cite{R}, where it is proved  for Sobolev spaces with more general $p$. Since we have made a slight modification of the result in \cite{R} to consider also the case of $\beta \ge n/2$, for the interested reader we give in the appendix some indications about its proof. 

\begin{remark} \label{remark2}
 As noticed already in \cite{BFRV10}, by the definition of the $\widetilde Q_{F,j}$ operator as an average in $\theta$ of  $\widetilde Q_{\theta,j}$, Proposition   \ref{prop.Qj} gives us an equivalent result for the convergence of the series  $ \sum_{j=l}^{\infty} \widetilde Q_{F,j}(q)$ (this follows from the fact that the estimates in \cite{R}  of the $\widetilde Q_{\theta,j}$ operators are uniform on $\theta$).
 \end{remark}
 
 We can now reduce the proof of Theorem \ref{teo.recovery} to proving Theorem $\ref{teo.Q2}$. 
\begin{proof}[Proof of Theorem $\ref{teo.recovery}$]
Taking the inverse Fourier transform of (\ref{eq.int.5}), we have that, {\it modulo a $C^\infty$ function}
\begin{equation} \label{eq.seriesT}
q(x) =  q_\theta(x) - \widetilde{Q}_{\theta,2}(q)(x)  -\sum_{j=3}^{\infty} { \widetilde{Q}_{\theta,j}(q)(x)}.
\end{equation}
By the previous proposition, if $\max(0,m) \le \beta< \infty$, the series on the left-hand side converges in the Sobolev space $W^{\alpha,2}(\RR^n)$ with $\alpha<\alpha_3$ given by (\ref{eq.alphaj}). On the other hand, when $\max(0,(n-3)/2) < \beta<\infty$, we have that $\alpha_3 \ge \min(1,\beta-(n-3)/2)$, which is the upper bound for the regularity of $\widetilde{Q}_{\theta,2}(q)$ given by (\ref{eq:rangepos}). Therefore Theorem \ref{teo.Q2} implies that, modulo a $C^{\infty}$ function, $q-q_\theta \in W^{\alpha,2}(\RR^n)$ with $\alpha$ satisfying (\ref{eq:rangepos}).
\end{proof}

To prove Theorem \ref{teo.Q2} we provide an explicit formula for $Q_{\theta,2}(q)$ in the Fourier transform side. 
Let $\zeta \in \RR^n$ and $r\in (0,\infty)$. We define the modified Ewald spheres
\begin{equation*} 
\Gamma_r(\zeta):=\{\xi\in\RR^n: |\xi-\zeta/2| = r|\zeta/2|\},
\end{equation*} 
and we denote by $\sigma_{rk}$ their Lebesgue measure.
\begin{proposition} \label{prop.fixangleformula}
Let $\theta \in \SP^{n-1}$ and $\eta \in \RR^n$. Then we have that
\begin{equation}  \label{eq:Bformula}
B_{\theta,2}(q)(\eta) = \chi_{H_\theta}(\eta) \left[ i \pi S_{\theta,1}(q)(\eta) + P_\theta(q)(\eta)\right],
\end{equation}
where
\begin{equation} \label{eq.S}
S_{\theta,r}(q)(\eta) :=  \frac{1}{k(r+1)}\int_{\Gamma_r(-2k\theta)}  \widehat{q}(\xi) \widehat{q}(\eta-\xi) \, d\sigma_{rk}(\xi),
\end{equation}
 with $k=k(\eta,\theta)$ given by $(\ref{eq.keta})$, $r\in (0,\infty)$,  and
\begin{equation} \label{eq.P}
{P_\theta(q)}(\eta) := \pv \int_{0}^\infty \frac{1}{1-r}{S_{\theta,r}(q)}(\eta) \,dr.
\end{equation}
\end{proposition}
\begin{proof}
The resolvent of the Laplacian satisfies the formula
\begin{equation*}
R_k(f)(x) = i \frac{\pi}{2}k^{n-2} \int_{S^{n-1}} \widehat{f}(k\omega) e^{ik x \cdot \omega} \, d\sigma(\omega) + \pv \int_{\RR^n} e^{i x \cdot \zeta} \frac{\widehat{f}(\zeta)}{-|\zeta|^2 + k^2} \, d\zeta.
\end{equation*}
This follows from computing the limit implicit in  (\ref{eq:resolvent}) in the sense of distributions; see, for  example, \cite[Chapter 5]{notasR} and \cite[pp. 209-236]{GS}. Using this formula in (\ref{eq.B}) for $j=2$, and computing the resulting Fourier transform in the $y$ variable, we get for $\eta \in H_\theta$,  $k= k(\eta,\theta)$, and $\theta' =\theta'(\eta,\theta)$ that
\begin{align*}
{B_{\theta,2}(q)}(\eta) &= i \frac{\pi}{2}k^{n-2}\int_{\RR^n} e^{-ik\theta' \cdot y}q(y) \int_{S^{n-1}} \widehat{q}(k\omega -k\theta) e^{ik y \cdot \omega} \, d\sigma(\omega) \,dy\\
& \hspace{15mm} +\int_{\RR^n} e^{-ik\theta' \cdot y}q(y) \pv \int_{\RR^n}  e^{i y \cdot \zeta} \frac{\widehat{q}(\zeta-k\theta)}{-|\zeta|^2 + k^2} \, d\zeta \,dy\\
&=  i \frac{\pi}{2}k^{n-2} \int_{S^{n-1}} \widehat{q}(k\theta'-k\omega) \widehat{q}(k\omega -k\theta) \, d\sigma(\omega)\\
&\hspace{15mm}+ \pv \int_{\RR^n}  \widehat{q}(k\theta'-\zeta) \frac{\widehat{q}(\zeta-k\theta)}{-|\zeta|^2 + k^2} \, d\zeta.
\end{align*}
If we put $\xi = k(\omega-\theta)$ in the  first integral and  $\xi = \zeta -k\theta$ in the second, then
\begin{equation*}
{B_{\theta,2}(q)}(\eta) = i \pi S_{\theta, 1}(q)(\eta) - \pv \int_{\RR^n}  \frac{\widehat{q}(\xi)\widehat{q}(\eta-\xi)}{\xi \cdot (\xi+2k\theta)} \, d\xi,
\end{equation*} 
when $\eta \in H_\theta$. Now, notice that
\[-\xi \cdot (\xi+2k\theta) =k^2-|\xi - (-k\theta)|^2 =(k-|\xi - (-k\theta)|) (k+|\xi - (-k\theta)|).\]
We  are going to take spherical coordinates with a radial parameter $t$ around the point $-k\theta$ in the principal value integral, denoting by $B_t$  the sphere of center $-k\theta$ and radius $t$, and by $\sigma_t$ its Lebesgue measure. Hence, if we use the change of  variables $t=rk$, $r \in (0,\infty)$,  in the radial variable, we obtain
\begin{align*}
\nonumber &-\pv \int_{\RR^n}  \frac{\widehat{q}(\xi)\widehat{q}(\eta-\xi)}{\xi \cdot (\xi+2k\theta)} \, d\xi  = \pv \int_{0}^\infty \frac{1}{\left(k-t \right)\left( k+t \right)}\int_{\partial B_t}  \widehat{q}(\xi) \widehat{q}(\eta-\xi) \, d\sigma_t(\xi) \,dt  \\
 &=  \pv \int_{0}^\infty \frac{1}{1-r} \frac{1}{k(1+r)} \int_{\Gamma_{r}(-2k\theta)}\widehat{q}(\xi) \widehat{q}(\eta-\xi) \, \ds{r}(\xi) \,dr \\
&= \pv \int_{0}^\infty \frac{1}{1-r} S_{\theta, r}(q)(\eta) \,dr. 
\end{align*}
\end{proof}

\section{Estimates of the spherical operator} \label{sec.3}

In order to  prove the estimates of the $\widetilde Q_{\theta,2}$ operator given by Theorem \ref{teo.Q2}, we are going to bound the $B_{\theta,2}$ operators, introduced in the previous section, with the help of Proposition \ref{prop.fixangleformula}.
We begin in this section  with estimates  for the spherical operator defined in (\ref{eq.S}). These estimates will be useful in the following section to bound the operator $P_\theta$ given in (\ref{eq.P}).
To simplify notation, as in (\ref{eq.cutoff2}), we define
\begin{equation}  \label{eq.SyPtilde} \widetilde S_{\theta,r}(q)(\eta) := {\chi}(\eta)S_{\theta,r}(q)(\eta), \hspace{4mm} \text{and} \hspace{4mm} \widetilde P_\theta(q)(\eta) := {\chi}(\eta)P_\theta(q)(\eta).
\end{equation}
\begin{lemma} \label{lemma.QspheL2}
Let $n\ge 2$ and  $q\in W_1^{\beta,2}(\RR^n)$ with $\beta\ge 0$.  Then if $0 \le \varepsilon<1$, the estimate
\begin{equation} \label{eq.PV1}
\norm{k^{\varepsilon} \widetilde S_{\theta,r}(q)}_{ L^2_\alpha}\le C{(1+r)^{-\gamma} } \norm{q}_{W_1^{\beta,2}}^2, 
\end{equation}
holds when
\begin{equation} \label{eq.range.Q2}
\begin{cases}
 \alpha \le \beta +(\beta- (n-3)/2)-\varepsilon, \hspace{3mm} if \hspace{3mm} (n-3)/2 +\varepsilon<\beta < (n-1)/2,\\
 \alpha<\beta +1-\varepsilon, \hspace{25mm} if \hspace{3mm} (n-1)/2  \le \beta<\infty ,
         \end{cases} 
\end{equation}
for some real number $\gamma >0$ (possibly depending on $\beta$).
\end{lemma}
To simplify later computations we define the operator  
\[ {\widetilde{K}_{r}(g_1,g_2)(\eta)} = {\chi}(\eta){K}_{r}(g_1,g_2)(\eta),\]
where
\begin{equation*} 
{{K}_{r}(g_1,g_2)}(\eta) := \frac{1}{k}\int_{\Gamma_{r}(-2k\theta)}|g_1(\xi)| |g_2(\eta-\xi)| \, \ds{r}(\xi).
\end{equation*}
Then we have that
\[
\left| \widetilde{S}_{\theta,r}(q)(\eta) \right| \le  \frac{1}{1+r}{\widetilde{K}_{r}(\widehat{q},\widehat{q})(\eta)},
\]
and therefore the proof of Lemma \ref{lemma.QspheL2} is an immediate consequence of the following lemma taking $\gamma =1- \lambda$. 
\begin{lemma} \label{lemma.K.Q2}
Let $n\ge 2$ and $f_1,f_2 \in W_1^{\beta,2}(\RR^n)$ with $\beta\ge 0$. Then if $0\le \varepsilon<1$, the estimate
\begin{equation} \label{eq.thm.sph.1}
\norm{k^\varepsilon \widetilde K_{r}(\widehat{f_1},\widehat{f_2})}_{ L^2_\alpha}\le C (1+r)^{\lambda} \norm{f_1}_{W_1^{\beta,2}} \norm{f_2}_{W_1^{\beta,2}},
\end{equation}
 holds when $\alpha$ is in the range given in $(\ref{eq.range.Q2})$, for some real number  $0<\lambda<1$ (possibly depending on $\beta$).
\end{lemma}
In the proof we are going to use the following result. 
\begin{lemma}\label{lemma.integrals}
Let $\SP_\rho\subset\RR^n$ be any sphere of radius $\rho$ and let $\sigma_\rho$ be its Lebesgue measure.
Then for any $0< \lambda\le (n-1)/2$, we have that
\begin{equation*} 
\int_{\SP_\rho} \frac{1}{|x-y|^{(n-1)-2\lambda}} \, d\sigma_\rho(y) \le C_\lambda \rho^{2\lambda},
\end{equation*}
 for any $x\in\RR^n$, and for a constant $C_\lambda$ that only depends on  $\lambda$. 
\end{lemma}
Its proof is  a  straightforward computation and we leave it for the appendix. We also need the following natural property of the Sobolev spaces.
\begin{remark} \label{remark3}
We have that $W^{\beta,2}_\delta \subset W^{\beta',2}_{\delta'}$ if $\beta \ge \beta'$ and $\delta \ge \delta'$. This follows from the equivalence
\begin{equation*} 
  \norm{\jp{\cdot}^\delta \jp{D}^{\alpha} f}_{L^2(\RR^n)} \sim \norm{ \jp{D}^{\alpha} \jp{\cdot}^\delta f}_{L^2(\RR^n)},
  \end{equation*}
  see, for example, \cite[p. 222]{BM09quad}.
\end{remark}
\begin{proof}[Proof of Lemma $\ref{lemma.K.Q2}$]
Consider a parameter $\varepsilon <\lambda <1$, for $\varepsilon$ in the statement, and observe that $k$ satisfies $|\eta| \le 2k$. Since $C_0>1$  in (\ref{eq.cutoff1}), we have that ${\chi}(\eta)=0$  for $|\eta| \le 1$. This means that  $|\eta|^{-1} \le  2 \jp{\eta}^{-1}$ in the region where $\chi$ does not vanish. Since $\lambda<1$, putting these inequalities together  we get  $k^{\lambda -1} \le C \jp{\eta}^{\lambda-1}$, and this yields
\begin{align*}
&\norm{k^\varepsilon \widetilde{K}_r({\widehat{f_1}}, \widehat{f_2})}_{L^2_\alpha}^2 \\
 &\le C \int_{\RR^n} \jp{\eta}^{2\alpha-2 +2\lambda} \left( k^{\varepsilon-\lambda}\int_{\Gamma_{r}(-2k\theta)} |\widehat{f_1}(\xi)| |\widehat{f_2}(\eta-\xi)| \,\ds{r}(\xi) \right)^2  d\eta.
\end{align*}
Now, we ask $\lambda$ to also satisfy the relation
\begin{equation} \label{eq.par.Q2}
\beta =  \alpha -1+\lambda.
\end{equation}
We have $\eta = (\eta-\xi) + \xi$, so if we choose any $0<c<1/2$, for every $\xi\in \Gamma_r(\eta)$ at least one of the conditions $|\xi|>c|\eta|$ and $|\eta-\xi|>c|\eta|$ must hold. Since we have assumed that $\beta \ge 0$, in both cases we are led, respectively, to the estimate
\begin{equation*}
\norm{k^\varepsilon \widetilde{K}_r({\widehat{f_1}}, \widehat{f_2})}_{L^2_\alpha}^2 \le C(I_1 + I_2), \hspace{5mm}\text{where}
\end{equation*}
\vspace{-5mm}
\begin{align*} 
I_1 &:=\int_{\RR^n}  \left( k^{\varepsilon-\lambda}\int_{\Gamma_r(-2k\theta)} |\widehat{f_1}(\xi)|\jp{\xi}^{\beta} |\widehat{f_2}(\eta-\xi)| \,\ds{r}(\xi) \right)^2 d\eta, \\
I_2 &:=\int_{\RR^n}  \left( k^{\varepsilon-\lambda}\int_{\Gamma_r(-2k\theta)} |\widehat{f_1}(\xi)| |\widehat{f_2}(\eta-\xi)|\jp{\eta-\xi}^{\beta} \,\ds{r}(\xi) \right)^2 d\eta,
\end{align*}
We study the case of $I_1$. Multiplying and dividing by $|\eta-\xi|^{(n-1)/2-\lambda +\varepsilon}$, and applying the Cauchy-Schwarz inequality, since $0\le \varepsilon<\lambda$ we have
\begin{align}
\nonumber I_1 \le C  \int_{\RR^n}  &\int_{\Gamma_r(-2k\theta)} |\widehat{f_1}(\xi)|^2\jp{\xi}^{2\beta}|\widehat{f_2}(\eta-\xi)|^2|\eta-\xi|^{n-1-2(\lambda-\varepsilon)}\,\ds{r}(\xi)   \times \\
\nonumber &\hspace{40 mm} \dots \times \int_{\Gamma_r(-2k\theta)}\frac{k^{-2(\lambda-\varepsilon)}}{|\eta-\xi|^{n-1-2(\lambda-\varepsilon)}}\,\ds{r}(\xi)\, d\eta \\
\label{eq.Q2sph.1}  \le C r^{2(\lambda-\varepsilon)}&\int_{\RR^n} \int_{\Gamma_{r}(-2k\theta)} |\widehat{f_1}(\xi)|^2{\jp{\xi}^{2\beta}}|\widehat{f_2}(\eta-\xi)|^2 |\eta-\xi|^{n-1-2(\lambda-\varepsilon)} \, \ds{r}(\xi)\, d\eta,
\end{align}
 where we need to impose the condition $\lambda \le (n-1)/2 +\varepsilon$, to apply Lemma \ref{lemma.integrals} and get the last inequality (recall $\Gamma_r(-2k\theta)$ has radius $rk$). To simplify the integral over the Ewald sphere we are going to use  the trace theorem. The fundamental point is that for spheres, the constant of the trace theorem can be taken to be $1$,  independently of the radius of the sphere. See Proposition \ref{prop:traceT} in the appendix for an elementary proof of this fact. This yields
\begin{align}
\nonumber \int_{\Gamma_{r}(-2k\theta)} &|\widehat{f_1}(\xi)|^2\jp{\xi}^{2\beta}|\widehat{f_2}(\eta-\xi)|^2 \jp{\eta-\xi}^{n-1-2(\lambda-\varepsilon)}\,\ds{r}(\xi)  \\
\nonumber  &\le  \int_{\RR^n}    |\widehat{f_1}(\xi)|^2\jp{\xi}^{2\beta}|\widehat{f_2}(\eta-\xi)|^2 \jp{\eta-\xi}^{(n-1)-2(\lambda-\varepsilon)}   \, d\xi \\
\label{eq.Q2sph.3} &+  \int_{\RR^n} \left| \nabla_\xi  \left(  \widehat{f_1}(\xi)\jp{\xi}^{\beta}\widehat{f_2}(\eta-\xi) \jp{\eta-\xi}^{(n-1)/2-(\lambda-\varepsilon)} \right) \right|^2 \, d\xi .
\end{align}
Therefore, inserting (\ref{eq.Q2sph.3}) into (\ref{eq.Q2sph.1}) and changing the order of integration we get
\begin{align}
\nonumber I_1 \le  \, Cr^{2(\lambda-\varepsilon)} &\int_{\RR^n} \int_{\RR^n} | \widehat{f_1}(\xi)|^2\jp{\xi}^{2\beta}|\widehat{f_2}(\eta-\xi)|^2 \jp{\eta-\xi}^{(n-1)-2(\lambda-\varepsilon)}  d\xi\,d\eta \\
\nonumber + Cr^{2(\lambda-\varepsilon)}\int_{\RR^n} &\int_{\RR^n} \left |\nabla \left( \widehat{f_1}(\xi)\jp{\xi}^{\beta}\right) \right |^2  |\widehat{f_2}(\eta-\xi)|^2 \jp{\eta-\xi}^{(n-1)-2(\lambda-\varepsilon)} d\xi \,d\eta\\
\nonumber + \, Cr^{2(\lambda-\varepsilon)} \int_{\RR^n} &\int_{\RR^n} | \widehat{f_1}(\xi)|^2\jp{\xi}^{2\beta}\left |\nabla \left(\widehat{f_2}(\eta-\xi) \jp{\eta-\xi}^{(n-1)/2-(\lambda-\varepsilon)}\right)\right|^2 d\xi\,d\eta\\
\label{eq.Q2sph.2}  &\le  Cr^{2(\lambda-\varepsilon)} \norm{f_1}_{W^{\beta,2}_1}^2\norm{f_2}_{W^{(n-1)/2-(\lambda-\varepsilon),2}_1}^2,
\end{align} 
since by  Plancherel theorem we have
\begin{equation*}
\int_{\RR^n} \left|\nabla (\widehat{f}(\xi)\jp{\xi}^{t} )\right|^2   \, d\xi \le C \norm{f}_{W^{t,2}_1}^2 .
\end{equation*}
The estimate of $I_2$ is nearly identical, the only difference is that we multiply and divide by the weight $|\xi|^{(n-1)/2 -(\lambda-\varepsilon)}$ in (\ref{eq.Q2sph.1}), so essentially we recover estimate (\ref{eq.Q2sph.2})  with the  roles of $f_1$ and $f_2$ interchanged.
 Hence we have that
\begin{align*}
&\norm{k^\varepsilon \widetilde{K}_r({\widehat{f_1},\widehat{f_2}})}_{L^2_\alpha} \\
&\le C r^{(\lambda-\varepsilon)} \left( \norm{f_1}_{W^{\beta,2}_1}\norm{f_2}_{W^{(n-1)/2-(\lambda-\varepsilon),2}_1} + \norm{f_2}_{W^{\beta,2}_1}\norm{f_1}_{W^{(n-1)/2-(\lambda-\varepsilon),2}_1} \right).
\end{align*}
Now, as a consequence of (\ref{eq.par.Q2}) and that $\lambda>\varepsilon$, equation (\ref{eq.thm.sph.1}) will follow directly in the range $\beta \ge (n-1)/2 $ (we are also using  remark \ref{remark3}). But, by the conditions imposed in the proof we have to take into account the restrictions
\begin{equation}\label{restriccion_1}
\begin{cases}
\varepsilon < \lambda < 1 \\ \varepsilon < \lambda \le \frac{n-1}{2}+\varepsilon
\end{cases}   \Longleftrightarrow \begin{cases}
\beta < \alpha < \beta +1 -\varepsilon \\ \beta +1 - \frac{n-1}{2}-\varepsilon \le \alpha < \beta +1 -\varepsilon.
\end{cases}                
\end{equation}
We can discard the lower bounds for  $\alpha$ using that $\norm{f}_{L^2_\alpha} \le \norm{f}_{L^2_{\alpha'}}$  always holds if $\alpha\le \alpha'$. Therefore  we have only the restriction $ \alpha<\beta +1 -\varepsilon$.

Otherwise, if $\beta$ is in the range $0 \le \beta < (n-1)/2 $, estimate (\ref{eq.thm.sph.1}) will follow if we add the extra condition  
\begin{equation} \label{rest2}
(n-1)/2-(\lambda-\varepsilon) \le \beta.
\end{equation}
Then, since $\lambda<1$, we must have $\beta>(n-3)/2 +\varepsilon$ (the other conditions on $\lambda$ don't add new restrictions). Also (\ref{eq.par.Q2}) and (\ref{rest2}) imply together  that $\alpha \le 2\beta-(n-3)/2-\varepsilon$, which is a stronger condition than $\alpha<\beta+1-\varepsilon$ since we have $\beta<(n-1)/2$. Hence, we have obtained the ranges of parameters given in the statement.
\end{proof}
In  the next section we are going to need the following Lipschitz estimate for $\widetilde S_{\theta,r}$ which follows from the previous lemma,  to bound the principal  value operator $\widetilde P_\theta$.
\begin{proposition} \label{prop.esdefnQ2}
Let $n\ge 2$ and $q\in \mathcal S(\RR^n)$. Then for any $0<\delta<1$ and $r_1,r_2\in (1-\delta,1+\delta)$,
\begin{equation} \label{eq.PV2}
\norm{ k^{-1} (\widetilde S_{\theta,r_1}(q)- \widetilde S_{\theta,r_2}(q))}_{ L^2_{\alpha}}\le C |r_1-r_2|   \norm{q}_{W^{\beta,2}_2}^2, 
\end{equation}
holds when $\alpha$ and $\beta$ satisfy $(\ref{eq.range.Q2})$ with $\varepsilon=0$.
\end{proposition}
In general the constant $C$ in the estimate is going to depend on $\delta$, but this has no special relevance. Observe also that the Sobolev space this time is $W^{\beta,2}_2$ instead of $W^{\beta,2}_1$.
\begin{proof}
We center the Ewald spheres in the origin with the change $\xi =  r k \omega -k\theta  $, where $\omega \in \SP^{n-1}$,
\begin{align*}
S_{\theta,r}(q)(\eta) &:= \frac{1}{k(1+r)} \int_{\Gamma_{r}(-2k\theta)}\widehat{q}(\xi) \widehat{q}(\eta-\xi) \, \ds{r}(\xi) \\
&=  \frac{2 k^{n-2} r^{n-1}}{(1+r)} \int_{\SP^{n-1}}\widehat{q}\left(r k \omega -k\theta \right) \widehat{q}\left(\eta-r k \omega +k\theta\right) \, d\sigma(\omega).
\end{align*} 
Now we can compute derivatives in the $r$ variable. Consider $\eta$ fixed, then
\begin{align*}
\nonumber &\frac{d\,}{dr}S_{\theta,r}(q)(\eta) \\
 \nonumber &=  \frac{ k^{n-2}((n-1)r^{n-2}(1+r)-r^{n-1})}{(1+r)^2} \int_{\SP^{n-1}}\widehat{q}\left( r k \omega -k\theta \right) \widehat{q}\left(\eta-r k \omega +k\theta \right) \, d\sigma(\omega) \\
\nonumber &+  \frac{ k^{n-1} r^{n-1}}{(1+r)} \int_{\SP^{n-1}}\omega \cdot\nabla \widehat{q}\left(r k \omega -k\theta \right) \widehat{q}\left(\eta-r k \omega +k\theta\right) \, d\sigma(\omega) \\
 &-   \frac{ k^{n-1} r^{n-1}}{(1+r)}  \int_{\SP^{n-1}} \widehat{q}\left(r k \omega -k\theta \right)\omega \cdot \nabla \widehat{q}\left(\eta-r k \omega +k\theta\right) \, d\sigma(\omega), 
\end{align*}
 (notice that $S_{\theta,r}(q)(\eta)$  is a smooth function in the $r$ variable for every $\eta\neq 0$). Hence, fixing some $0<\delta<1$,  for $r\in (1-\delta,1+\delta)$, if we undo the change to spherical coordinates we get
\begin{equation} \label{eq.lema.esdefnQ2.3}
  \left | \frac{d\,}{dr}  S_{\theta,r}(q)(\eta) \right | \le C  K_{r}(\widehat{q},\widehat{q})(\eta) + C k  K_{r}(|\nabla\widehat{q}|,\widehat{q})(\eta) + Ck  K_{r}(\widehat{q},|\nabla \widehat{q}|)(\eta),
 \end{equation}
taking the absolute values inside the integrals. Notice the $k$ factor multiplying the last  terms. Now, if $\eta\neq 0$, by the fundamental theorem of calculus we have
\begin{align*}
\nonumber S_{\theta,r_2}(q)(\eta) - S_{\theta,r_1}(q)(\eta) &=  \int^{r_2}_{r_1}  \frac{d\,}{dr}  S_{\theta,r}(q)(\eta) \, dr \\
 &= (r_2-r_1) \int_0^1  \left[\frac{d\,}{dr}  S_{\theta,r}(q)(\eta)\right]_{r= r(t)}  dt,
\end{align*}
where for brevity, $r(t) = (r_2-r_1)t + r_1$. Then by (\ref{eq.lema.esdefnQ2.3}) we obtain
\begin{align*} 
&|\nonumber {S}_{\theta,r_2}(q)(\eta) - {S}_{\theta,r_1}(q)(\eta)|\\ 
&\hspace{5mm}\le C   |r_2-r_1| \int_0^1  \left ( K_{r(t)}({\widehat{q}},\widehat{q})(\eta) + k K_{r(t)}(|\nabla\widehat{q}|,\widehat{q})(\eta)+k K_{r(t)}(\widehat{q},|\nabla\widehat{q}|)(\eta) \right)  \,dt,
\end{align*}
so multiplying by ${\chi}(\eta)$,  and applying Minkowski's integral inequality we have
\begin{align*}
  &\norm{ k^{-1}(\widetilde S_{\theta,r_2}(q) - \widetilde S_{\theta,r_2}(q))}_{L^2_{\alpha}} \\
   &\le C|r_2-r_1|  \int_0^1 \hspace{-1mm}\left( \norm{ k^{-1}\widetilde{K}_{r(t)}(\widehat{q},\widehat{q}) }_{L_{\alpha}^2}+  \norm{ \widetilde K_{r(t)}(|\nabla\widehat{q}|,\widehat{q})}_{L_{\alpha}^2} + \norm{ \widetilde K_{r(t)}(\widehat{q}, |\nabla \widehat{q}|)}_{L_{\alpha}^2} \right) dt.
\end{align*}
 Then, since $r(t) \in (1-\delta,1+\delta)$, we can apply Lemma \ref{lemma.K.Q2} with $\varepsilon=0$ to estimate the first term (using that $k^{-1} \le 2|\eta|^{-1}$). The others follow similarly. Observe that
\begin{equation*}
 K_{r}(|\nabla\widehat{q}|,\widehat{q})\le C \sum_{i=1}^n K_{r}(\partial_i \widehat{q},\widehat{q})= C \sum_{i=1}^n K_{r}(\widehat{x_iq},\widehat{q}),
 \end{equation*}
so again we can apply Lemma \ref{lemma.K.Q2} to estimate these terms, to obtain
\begin{equation*} 
\norm{\widetilde K_{r}(\widehat{x_iq},\widehat q)}_{ L^2_\alpha}\le C  \norm{x_iq}_{W_1^{\beta,2}} \norm{q}_{W_1^{\beta,2}} \le C  \norm{q}_{W_2^{\beta,2}}^2. 
\end{equation*}
This completes the proof.
\end{proof}

%%%%%%%%%%%%%%%%%%%%%%%%%%%%%%%%%%%%%%%%%%%%%%%%%%%%%%%%%%%%%%%%%%%%%%%%
%%%%%%%%%%%%%%%%%%%%%%%%%%%%%%%%%%%%%%%%%%%%%%%%%%%%%%%%%%%%%%%%%%%%%%%%
%%%%%%%%%%%%%%%%%%%%%%%%%%%%%%%%%%%%%%%%%%%%%%%%%%%%%%%%%%%%%%%%%%%%%

\section{Estimate of the Principal Value Operator} \label{sec.4}

As a consequence of Lemma \ref{lemma.QspheL2} and Proposition \ref{prop.esdefnQ2}, we obtain the following estimate for the principal value operator $\widetilde P_\theta$ introduced in (\ref{eq.SyPtilde}).
\begin{proposition} \label{prop.PV}
Let $n\ge 2$ and $q\in W_2^{\beta,2}(\RR^n)$ with $\beta\ge 0$. Then the estimate
\begin{equation*} 
\norm{\widetilde P_\theta(q)}_{ L^2_{\alpha}}\le C  \norm{q}_{W_2^{\beta,2}}^2,
\end{equation*}
holds  when
\begin{equation} \label{eq.range.P}
\alpha < \begin{cases}
  \beta +(\beta- (n-3)/2), \hspace{3mm} if \hspace{3mm} (n-3)/2 <\beta < (n-1)/2,\\
 \beta +1, \hspace{25mm} if \hspace{3mm} (n-1)/2 \le \beta<\infty.
         \end{cases} 
\end{equation}
\end{proposition}
\begin{proof}
By a density argument we might assume $q\in \mathcal S(\RR^n)$.   Take the same $\delta$ as in Proposition \ref{prop.esdefnQ2}, 
\begin{align*}
\nonumber \widetilde{P}_\theta(q)(\eta)&=P.V.\int_{|1-r|<\delta} \frac{\widetilde{S}_{\theta,r}(q)(\eta)}{1-r}dr +\int_{\delta<|1-r|} \frac{\widetilde{S}_{\theta,r}(q)(\eta)}{1-r}\,dr \\
    &:= P_{\theta,\delta}(q)(\eta)+P_{\theta,L}(q)(\eta),
\end{align*}
 (we drop the tilde symbol from the operators just defined to simplify notation). Applying Minkowski's integral inequality and estimate (\ref{eq.PV1}) with $\varepsilon=0$, we obtain
\begin{equation}\label{eq.PV.1}
\|P_{\theta,L} (q)\|_{L^2_\alpha} \le  \int_{  \delta<|1-r| } \frac{\|\widetilde{S}_{\theta,r}(q)\|_{L^2_\alpha} }{|1-r|} \, dr \leq C \norm{q}_{W_2^{\beta,2}}^2,
\end{equation} 
To study $P_{\theta,\delta}$ we need a finer decomposition in regions. Set $\delta_k :=  \delta\min \left(k^{-2},1 \right)$. To simplify notation we define the region
\[{B_k := \{r\in (0,\infty): \delta_k \le |1-r| \le \delta \},}\]
relevant when $k>1$. By using that $P.V.\int_{|1-r|<a}\frac{dr}{1-r}=0$ for any $a>0$, we have
\begin{align}
\nonumber P_{\theta,\delta}(q)(\eta) &= \int_{|1-r|<\delta_k} \frac{\widetilde{S}_{\theta,r}(q)(\eta)-\widetilde{S}_{\theta,1}(q)(\eta)}{1-r}dr  + \int_{B_k} \frac{\widetilde{S}_{\theta,r}(q)(\eta)}{1-r}\,dr \\
  \label{eq.noPV2}  &:= P_{\theta,\delta_k}(q)(\eta)+P_{\theta,B_k}(q)(\eta),
\end{align}
where in (\ref{eq.noPV2}) the $P.V.$ is no longer necessary since $q\in \mathcal S$ implies that $\widetilde S_{\theta,r}(q)(\eta)$ is smooth  in the $r$ variable, so the singularity in the denominator is cancelled by the numerator. Then, using Cauchy-Schwarz's inequality in the $r$ variable and estimate (\ref{eq.PV2}) we obtain
\begin{align}
\nonumber &\norm{P_{\theta,\delta_k}(q)}_{ L^2_{\alpha}}^2 = \int_{\RR^n} \jp{\eta}^{2\alpha}
\left |\int_{|1-r| <  \delta_k} \frac{\widetilde{S}_{\theta,1}(q)(\eta)-\widetilde{S}_{\theta,r}(q)(\eta)}{1-r}  \,dr\right|^2 d\eta \\
 \nonumber &\le 2 \delta \int_{\RR^n} \jp{\eta}^{2\alpha} \int_{|1-r| <  \delta_k}
\left( k^{-1}\frac{|\widetilde{S}_{\theta,1}(q)(\eta)-\widetilde{S}_{\theta,r}(q)(\eta)| }{|1-r|}  \right)^2 dr  \,d\eta \\
 \label{eq.PV.2} & \le 2 \delta \int_{|1-r| <  \delta} \frac{\norm{k^{-1}(\widetilde{S}_{\theta,1}(q)(\eta)-\widetilde{S}_{\theta,r}(q)(\eta))}_{L_\alpha^{2} }^2}{|1-r|^2} \,dr     \le C \delta^2\norm{q}_{W_2^{\beta,2}}^4.
\end{align}
Now, to estimate $P_{\theta,B_k}(q)(\eta)$, set $N(k) = - \log_2(\delta k^{-2})$, and consider the next dyadic decomposition,
\begin{equation*}
P_{\theta,B_k}(q)(\eta) = \sum_{0 \le j<N(k)}  \int_{\{2^{-(j+1)}<|1-r|<2^{-j}\}} \chi_{B_k}(r)  \frac{1}{1-r} {\widetilde{S}_{\theta,r}(q)}(\eta) \,dr ,
\end{equation*}
 where $\chi_{B_k}$ is the characteristic function of ${B_k}$. For $\eta$ fixed, if $0\le j <N(k)$, the definition of $N(k)$ implies that ${2^j \le  k^{2}}/\delta$. Therefore
 \begin{equation}  \label{eq.diadic1}
|P_{\theta,{B_k}}(q)(\eta)|\le \sum_{j=0}^\infty 2^{j+1} \chi_{(\delta 2^{j},\infty)}(k^{2})\int_{|1-r|<2^{-j}} |{\widetilde{S}_{\theta,r}(q)}(\eta)| \,dr ,
\end{equation}
where  $\chi_{(\delta 2^{j},\infty)}$ is again a characteristic function. But observe that in the last line we have a sublinear operator of the kind
\[{P^{\lambda}(q)}(\eta) := \chi_{({\delta \lambda^{-1}},\infty)} (k^2) \int_{|1-r|\le \lambda} |{\widetilde{S}_{\theta,r}(q)}(\eta)| \, dr,\]
 with $0<\lambda<1$. Take $\varepsilon>0$ small. Computing the $L^2_{\alpha}$ norm of $P^{\lambda}$ and applying Minkowski's integral inequality we obtain
\begin{equation}  \label{eq.diadic2}
\norm{P^{\lambda}(q)}_{ L^2_{\alpha}}  \le C{\lambda}^{\varepsilon/2}  \int_{|1-r|\le {\lambda} } \norm{k^\varepsilon \widetilde{S}_{\theta,r}(q)}_{ L^2_\alpha} \,dr \le {\lambda}^{1+\varepsilon/2}  C  \norm{q}_{W_2^{\beta,2}}^2,
\end{equation}
using estimate (\ref{eq.PV1}), and that in the region where the characteristic function does not vanish we have that $k^{-\varepsilon} \le C \lambda^{\varepsilon/2}$. Hence, taking the $L^2_{\alpha}$ norm of (\ref{eq.diadic1}) and applying estimate (\ref{eq.diadic2}),
\begin{equation} \label{eq.PV.3}
\norm{P_{\theta,{B_k}}(q)}_{ L^2_{\alpha}} \le 2 \sum^{\infty}_{j=0}  2^{j} \norm{ {P^{2^{-j}}(q)} }_{ L^2_{\alpha}} 
\le    C \norm{q}_{W_2^{\beta,2}}^2 \sum^{\infty}_{j=0}  2^{-j \varepsilon/2},
\end{equation}
and the dyadic sum converges. This estimate holds when $\alpha$ is in the range given by (\ref{eq.range.Q2}), so for every $\alpha$ in the range given by (\ref{eq.range.P}) is possible to chose $\varepsilon$  so that (\ref{eq.PV.3}) holds.  Therefore since $\widetilde{P}_\theta = P_{\theta,L}+P_{\theta,\delta_k}+P_{\theta,{B_k}}$ we conclude the proof putting together estimates (\ref{eq.PV.1}), (\ref{eq.PV.2}) and (\ref{eq.PV.3}). 
\end{proof}

\section{Some limitations on the regularity of the double dispersion operator}  \label{sec.5}
In this section we construct a family of real, radial and compactly supported functions $g_\beta$ to obtain upper bounds  for the regularity gain of the $Q_{\theta,2}$ and $Q_{F,2}$ operators.  We will also give the proofs of Theorems \ref{teo.1} and \ref{teo.Q2}.
  \begin{theorem} \label{teo.count}
For every $0<\beta<\infty$, if $\alpha_0:= \min(\beta + 1, 2\beta  -(n-4)/2)$, there is a radial, real and compactly supported function $g_\beta$ satisfying  $g_\beta \in W^{\gamma,2}(\RR^n)$ if $\gamma<\beta$, and such that
\begin{enumerate}[i)]
\item $ Q_{\theta,2}(g_\beta) \in W_{loc}^{\alpha,2}(\RR^n)$ only if $\alpha< \alpha_0$.
\item $ Q_{F,2}(g_\beta) \in W_{loc}^{\alpha,2}(\RR^n)$ only if $\alpha< \alpha_0$.
\end{enumerate}
\end{theorem}

To simplify notation, we put $S_\theta :=S_{\theta,1}$, $\Gamma(\zeta):=\Gamma_1(\zeta)$ and denote by $\sigma_k$ the Lebesgue measure of this Ewald sphere. 
The key idea behind the proof of this theorem is to  study the asymptotic behavior of $|{ {B_{\theta,2}}(q)}(\eta)|$ when $|\eta| \to \infty $ and $q=g_\beta$. Now,  by construction (see Proposition \ref{prop:constr} below), $g_\beta$ has  real and nonnegative Fourier transform. Hence, by (\ref{eq:Bformula}),  for  $\eta\in H_\theta$, ${B_{\theta,2}(g_\beta)}(\eta)$ has a real part given by $P_\theta(g_\beta)(\eta)$ and an imaginary part given by $ \pi S_\theta(g_\beta)(\eta)$ (see Proposition \ref{prop.fixangleformula}). Therefore, since there are no possible cancellations between the real and imaginary parts,  we are going to study  the asymptotic behavior  of only the spherical integral $S_\theta(g_\beta)$, which has the advantage of having a positive integrand.  The key estimate is the following.
\begin{lemma} \label{lemma.lowBound}
Consider the half cone  $D_\theta := \{ \eta \in \RR^n :\eta \cdot \theta \le -a |\eta|\}$ for some $0<a<1$. Assume also that $q_\beta \in L^2(\RR^n)$ satisfies the following conditions,
\begin{enumerate}[i)]
\item Its Fourier transform $\widehat{q_\beta}(\xi)$ is continuous, real and nonnegative in all $ \RR^n$.

\item There is a constant $c>0$ such that if $|\xi|>c$, then $\widehat{q_\beta}(\xi)\ge C\jp{\xi}^{-n/2-\beta}$.

\item   $\widehat{q_\beta}(0) >0$.
\end{enumerate}
Then we  have that if  $|\eta|>4c$, there is a constant $C$ independent of $\eta$ and $\theta$ such that
\begin{equation*}  
S_\theta(q_\beta)(\eta) \ge  C \chi_{D_\theta}(\eta) \max \left(  \jp{\eta}^{-\beta-n/2 -1} , \jp{\eta}^{-2\beta-2 }\right ) ,
\end{equation*}
where $\chi_{D_\theta}$ denotes the characteristic function of the cone.
\end{lemma}

\begin{proof} 
Observe that if $\eta \cdot \theta  \le -a |\eta|$, by (\ref{eq.keta}) we have that $k \sim |\eta|$. Since $\widehat{q_\beta}$ is nonnegative, we have that
\begin{equation} \label{eq.example.2}
 S_\theta(q_\beta)(\eta) \ge   \frac{1}{2k}\int_{A(\eta)}\widehat{q_\beta}(\xi)\widehat{q_\beta}(\eta-\xi) \,d\sigma_k(\xi),
\end{equation}
where $A(\eta)\subset \Gamma(-2k\theta)$ is defined  as follows
\[A(\eta) := \{ \xi \in \Gamma(-2k\theta): |\xi|>c \, \, \, \text{and} \, \, \, |\eta-\xi|>c \}.\]
That is, we take the points on the Ewald sphere which are not contained in two  balls of radius $c$ centered, respectively, around $\eta$ and the origin. This implies that the measure of $A(\eta)$ satisfies $|A(\eta)| \ge Ck^{n-1}$  for some constant $C>0$, since $|\eta|>4c$ implies $k>2c$.  Therefore, by condition $ii)$, in (\ref{eq.example.2}) we get
\begin{align}
\nonumber S_\theta(q_\beta)(\eta) &\ge C\frac{1}{k}\int_{A(\eta)} \jp{\eta-\xi}^{-\beta-n/2}\jp{\xi}^{-\beta-n/2} \,d\sigma_k(\xi)\\
  \label{eq.example.3} &\ge  C \frac{1}{k}\jp{\eta}^{-2\beta-n} k^{n-1} \ge C \jp{\eta}^{-2\beta -2},
\end{align}
where  we have used that $|\xi|\le 2k \le C|\eta|$ and  that ${|\eta-\xi| \le |\eta| + 2k \le C|\eta|}$. 

 Now, if $\widehat{q_\beta}$ is continuous  and $\widehat{q_\beta}(0)>0$, we can take a ball $B_\varepsilon$ around the origin of radius $0 < \varepsilon< c$  such that $\widehat{q_\beta}{(\xi)}$ is positive in its closure. Then if $|\eta|> 2c$, $\xi \in B_\varepsilon\cap \Gamma(-2k\theta)$ implies $|\eta-\xi| > c$, so
\begin{align}
\nonumber S_\theta(q_\beta)(\eta) &\ge  \frac{1}{k}\int_{B_\varepsilon\cap \Gamma(-2k\theta)}\widehat{q_\beta}(\xi)\widehat{q_\beta}(\eta-\xi) \,d\sigma_k(\xi)\\
\label{eq.example.4} &\ge C \frac{1}{k}\int_{B_\varepsilon\cap\Gamma(-2k\theta)} \jp{\eta-\xi}^{-\beta-n/2} \,d\sigma_k(\xi) \ge  C \jp{\eta}^{-\beta-n/2 -1} ,
\end{align}
using again that  $|\eta-\xi| \le C|\eta|$ and that the measure $|B_\varepsilon\cap\Gamma(-2k\theta)|$ is bounded below by a positive constant independent of $\eta$ (this is because the region $B_\varepsilon\cap\Gamma(-2k\theta)$ approaches a flat disc of radius $\varepsilon$ for $\eta$ large). To finish we just have  to put together (\ref{eq.example.3}) and (\ref{eq.example.4}).
\end{proof}
We construct now the family of functions $g_\beta$.
  \begin{proposition}\label{prop:constr}
For every $0<\beta<\infty$ there is a radial, real and compactly supported function $g_\beta \in W^{\gamma,2}(\mathbb{R}^n)$ for every  $\gamma<\beta$, such that $\widehat{g_\beta}$ is nonnegative in $\mathbb{R}^n$, $\widehat{g_\beta}(0)>0$, and for  some $c>0$ we have that
\begin{equation}  \label{eq:gbLowerbound}
\widehat{g_\beta}(\xi) \ge  C <\xi>^{-n/2-\beta} \text{\, \, if \, \,} |\xi|>c.
\end{equation}
\end{proposition}
\begin{proof}
We introduce the functions $G_\beta(x)$ given  by the relation
\[\widehat{G_\beta}(\xi) := \frac{1}{ <\xi>^{n/2+\beta}}.\] 
These functions are,  up to normalizing factors,  kernels of Bessel potential operators. We observe that the Fourier transform of a radial and real function in $\mathbb{R}^n$ is also radial and real. 
As a consequence, the  $G_\beta$ functions satisfy the  statement of the proposition except for the condition of compact support.

The regularity properties of the $G_\beta$ function are determined by its behavior when $|x|\to 0$. Far from the origin $G_\beta(x)$ is smooth with exponential decay (see, for example, chapter V of \cite{stein}).  This motivates us to choose $g_\beta = \phi G_\beta$ where, $\phi$ is any $C^\infty_c(\mathbb{R}^n)$ function nonvanishing at the origin. Then clearly we have $g_\beta \in W^{\gamma,2}(\mathbb{R}^n)$ for every $\gamma<\beta$, as desired.

 The rest of the properties of $g_\beta$ follow if we choose $\phi$ in the following way. Consider again a  radial and real function $\psi\in C^\infty_c(\mathbb{R}^n)$ such that $\widehat{\psi} (0) \neq 0$, and put  $\phi= \psi*\psi$. Then $\phi$ is going to be compactly supported, radial, real, and nonzero at the origin. Moreover its Fourier transform satisfies $\widehat \phi(\xi)= \widehat \psi(\xi)^2 \ge 0$ for all $\xi \in \mathbb{R}^n$ and also that $\widehat{ \phi}(0) >0$. 
 
Using this we get that  $\widehat{g_\beta}(\xi) = \widehat{\phi}*\widehat{G_\beta}(\xi)$ is  nonnegative. Also, since $\widehat{ \phi}(0) >0$, there is an $\varepsilon>0$ such that $\widehat{\phi}(\xi)$ is bounded below in $B_\varepsilon = \{ \xi \in \mathbb{R}^n : |\xi|<\varepsilon\}$.
 This yields \eqref{eq:gbLowerbound} since we have 
 \begin{multline*}
\widehat{g_\beta}(\xi) =  \int_{\mathbb{R}^n} \widehat{G_\beta}(\xi-\zeta)\widehat \phi(\zeta)\,d\zeta \ge  \int_{B_\varepsilon} <\xi-\zeta>^{-n/2-\beta} \widehat \phi(\zeta)\,d\zeta \\
 \ge C \int_{B_\varepsilon} <\xi-\zeta>^{-n/2-\beta} \, d\zeta  \ge C <|\xi|+\varepsilon>^{-n/2-\beta} \ge C <\xi>^{-n/2-\beta},
 \end{multline*}
for $|\xi|> \varepsilon$. 
To finish the proof we only have to verify that $\widehat{g_\beta}(0)>0$. But this is immediate,
\begin{equation*}
\widehat{g_\beta}(0) = \int_{\mathbb{R}^n} \widehat{G_\beta}(-\xi)\widehat \phi(\xi)\,d\xi>0,
\end{equation*} 
since $\widehat{G_\beta}(\xi)>0$  and $\widehat \phi(\xi)\ge 0$ for every $\xi\in \mathbb{R}^n$.
\end{proof}

We can now prove  Theorem \ref{teo.count}, with the help of the following simple result.
\begin{lemma} \label{lemma:Wloc}
Let $D_\theta$ be a half cone as  defined  in Lemma $\ref{lemma.lowBound}$, and  let $f \in \mathcal S'(\RR^n)$ be such that $\widehat f \in L^1_{loc}(\RR^n)$ and  $\widehat{f}(\eta)\ge 0$ almost everywhere. Assume also that for some $c>0$, $\gamma \in \RR $ and $|\eta|>c$ we have   ${\widehat{f}(\eta)\ge C \jp{\eta}^{-n/2-\gamma}}$ for $\eta \in {D_\theta}$. Then we have that   $f\notin W_{loc}^{\alpha,2}$ if $\alpha \ge \gamma$.
\end{lemma}
We postpone the proof of this lemma until the end of this section.
\begin{proof}[Proof of Theorem $\ref{teo.count}$]
By Proposition \ref{prop:constr}, we have that $g_\beta \in W^{\gamma,2}$ if and only if $\gamma<\beta$. Let's   prove $i)$. $g_\beta$ satisfies all of the conditions necessary to apply Lemma \ref{lemma.lowBound}, and  hence, for $|\eta|>4c$, we have that
\begin{equation}\label{eq.max1}
S_\theta(g_\beta)(\eta) \ge C \chi_{D_\theta}(\eta) \max \left(  \jp{\eta}^{-\beta-n/2 -1} , \jp{\eta}^{-2\beta-2 }\right ).
 \end{equation}
 As we mentioned before, since $g_\beta$ is real, there are no cancellations possible between $P_\theta(g_\beta)$ and $i\pi S_\theta(g_\beta)$. Hence if we assume $Q_{\theta,2}(g_\beta) \in W_{loc}^{\alpha,2}$, it implies that $\mathcal F^{-1} (S_\theta(g_\beta)) \in W_{loc}^{\alpha,2} $. As a consequence, by Lemma \ref{lemma:Wloc} and (\ref{eq.max1}) we obtain that $\alpha$ must satisfy  simultaneously $\alpha< \beta+1$ and  $\alpha< 2\beta+ (n-4)/2$.

To prove $ii)$ observe that taking the imaginary part of (\ref{eq.BjF}), we have
\begin{equation*} 
 \mathcal{I} (\widehat{Q_{F,2}(g_\beta)})(\eta) =  \frac{2\pi}{|\SP^{n-1}|}\int_{ \{\theta \in S^{n-1}: \, \eta \cdot \theta<0\}} S_\theta(g_\beta)(\eta) \, d\sigma(\theta).
\end{equation*}
 Therefore, since the integrand is positive for every $\theta$, if we consider  an $\eta$  fixed satisfying $|\eta|>4c$, we can restrict the integral to the subset of points  $\theta \in \SP^{n-1}$ such that $\eta \in D_\theta$. Then we obtain
 \begin{align*} 
 \mathcal{I} (\widehat{Q_{F,2}(g_\beta)})(\eta) &\ge  \frac{2 \pi }{|\SP^{n-1}|} \int_{\{ \theta \in S^{n-1}: \, \, \eta \cdot \theta < - a |\eta| \}} { S_\theta(g_\beta)}(\eta) \, d\sigma(\theta) \\
 &\ge C \max \left(  \jp{\eta}^{-\beta-n/2 -1} , \jp{\eta}^{-2\beta-2 }\right ),
\end{align*}
where the last line follows from (\ref{eq.max1}). From this estimate, reasoning as in the proof of $i)$ we obtain $ii)$.
\end{proof}
 We can finally prove the remaining main theorems of this work.
 \begin{proof}[End of proof of theorem $\ref{teo.Q2}$]
  Condition (\ref{eq.remarkable}) is an immediate consequence of Theorem \ref{teo.count}.
By Proposition \ref{prop.fixangleformula} and equations (\ref{eq.QjB}) and (\ref{eq:Bformula}), we have that 
\[\widehat{\widetilde Q_{\theta,2}(q)}(\eta)= \chi_{H_\theta}(\eta) \left(\widetilde S_\theta(q)(\eta) + \widetilde P_\theta(q)(\eta)\right) + \chi_{H_{-\theta}}(\eta) \left(\widetilde S_{-\theta}(q)(\eta) + \widetilde P_{-\theta}(q)(\eta)\right),\]
so, the estimate of the spherical operators follows from Lemma $\ref{lemma.QspheL2}$ with $\varepsilon=0$ and $r=1$, and the estimate of the principal value operators from Proposition \ref{prop.PV}.
\end{proof}
\begin{proof}[Proof of Theorem $\ref{teo.1}$]
Take $\alpha\ge 0$ and assume that we have that $q-q_\theta \in W^{\alpha,2}_{loc}(\RR^n)$ for every compactly supported, real and radial potential $q\in W^{\beta,2}(\RR^n)$. We are going to prove that then necessarily $Q_{\theta,j}(q) \in W^{\alpha,2}_{loc}(\RR^n)$ also.

Consider the Born series in (\ref{eq.seriesT}), for the potential $q(\lambda) = \lambda q$, where $\lambda \in (0,1)$, and denote by $q_\theta(\lambda)$ its corresponding Born approximation. By the multilinearity  of the $\widetilde Q_{\theta,j}$ operators we have
\begin{equation} \label{eq.algebra.1}
{\lambda q}- { q}_{\theta}(\lambda) =  - \sum_{j=2}^{l-1} \lambda^{j}{\widetilde Q_{\theta,j}(q)} + \sum_{j=l}^{\infty} \lambda^{j}{\widetilde Q_{\theta,j}(q)},
\end{equation}
modulo a $C^\infty$ function (possibly  dependent on $\lambda$).
 By Proposition \ref{prop.Qj}, we have that if $\beta \ge m$, we can take  $l$ in (\ref{eq.algebra.1}) such that $\alpha< \alpha_l $, and hence $\sum_{j=l}^{\infty} \lambda^{j}{\widetilde Q_{\theta,j}(q)} $   will converge in $W^{\alpha,2}(\RR^n)$, for every $\lambda \in (0,1)$.
 
Let $V := \mathcal{S}'/ W^{\alpha,2}_{loc}$ be the quotient vector space of $\mathcal{S}'(\RR^n)$ with $W^{\alpha,2}_{loc}(\RR^n)$.  We denote the elements of $V$ by $[h]$ where $h\in \mathcal{S}'$ is any member of the equivalence class. Since by hypothesis we have that ${\lambda q}- { q}_{B}(\lambda)$ is  a $W_{loc}^{\alpha,2}(\RR^n)$ function,  then (\ref{eq.algebra.1}) becomes in $V$,
\[ \sum_{j=2}^{l-1} \lambda^{j}\left [ \widetilde Q_{\theta,j}(q)\right] = 0.\]
 Now, for every $2 \le  i \le l-1$, we can always choose a $\lambda_i \in (0,1)$  such that the $l-2$ vectors of $\RR^{l-2}$ with coordinates $(\lambda_i^2,\dots,\lambda_i^{l-1})$ are linearly independent. This implies that  $[\widetilde Q_{\theta,j}(q) ]= 0$, and hence, that $\widetilde Q_{\theta,j}(q) \in W^{\alpha,2}_{loc}(\RR^n)$ for all $2 \le j \le l-1$. As a consequence, we also obtain that $ Q_{\theta,j}(q) \in W^{\alpha,2}_{loc}(\RR^n) $ for every $j$. But by Theorem \ref{teo.Q2}, we know that this implies that  $\alpha \le \min(2\beta-(n-4)/2, \beta +1)$. 
 
 Using  Remark \ref{remark2} and Theorem \ref{teo.count} the case of full data scattering can be proved in the same way.
\end{proof}

\begin{proof}[Proof of Lemma $\ref{lemma:Wloc}$]
In  the proof of Proposition \ref{prop:constr} we have seen that we can take a function $\phi \in C^\infty_c$ such that $\widehat \phi(\xi) \ge 0$ in $\RR^n$ and $\widehat \phi(0)>0$. Then we can choose an $0<\varepsilon<c$ small so that $\widehat{\phi}(\xi)$ is bounded below by a positive constant when $|\xi|<\varepsilon$. Then, if $ D_{\theta,\varepsilon} :=\{ \eta \in D_\theta : (\eta-\xi) \in D_\theta \, \forall \,|\xi|<\varepsilon\}$ and we ask $ |\eta| \ge 2c$ and $\eta \in D_{\theta,\varepsilon}$, we obtain that
\begin{align*}
\widehat{\phi f}(\eta) &= \int_{\RR^n} \widehat \phi(\xi) \widehat{f}(\eta-\xi) \, d\xi \\
&\ge \int_{B_\varepsilon} \widehat \phi(\xi) \widehat{f}(\eta-\xi) \, d\xi \ge C\jp{\eta}^{-n/2-\gamma}.
\end{align*}
As a consequence we have that $\phi f\notin W^{\alpha,2}(\RR^n)$ for $\alpha\ge\gamma$, which implies that $f\notin W_{loc}^{\alpha,2}(\RR^n)$ by  definition of the local Sobolev spaces.
\end{proof}

\section*{Appendix}
 \setcounter{equation}{0}
\renewcommand{\theequation}{A.\arabic{equation}}
 \setcounter{theorem}{0}
     \renewcommand{\thetheorem}{A.\arabic{theorem}}
     
 We begin giving some indications about how to prove Proposition \ref{prop.Qj} from  analogous results in \cite{R}.
 \begin{proof}[Proof of Proposition $\ref{prop.Qj}$]
 In the proof of Proposition 4.1 of \cite{R}, using estimates for the resolvent $R_k$, it is shown that for compactly supported $q$
\begin{equation} \label{eq.Af}
\norm{\widetilde Q_{\theta,j}(q)}_{W^{\alpha,2}}^2 \le  C^{2j} \norm{q}_{W^{\beta,2}}^{2j} \int_{C_0}^{\infty} k^{n-1 + 2\alpha -2\gamma} \, dk,
\end{equation}
(notice that in \cite{R}, the operators are indexed differently, with $j' =j-1$) where $C=C(n,\alpha,\beta,\supp \, q)>0$, 
\begin{equation} \label{eq.Afprima}
\gamma =  \beta-(j-1) - \frac{(n-1)}{2}  \left ( \frac{1}{t_{1}} - \sum_{\ell =1}^{j-1} \left(\frac{1}{r_{\ell}} - \frac{1}{t_{\ell +1}}\right)  -\frac{3}{2} \right),
\end{equation}
and $t_\ell,r_\ell$,  are parameters that for $\ell= 1, \dots, j-1$ must satisfy the conditions $t_{\ell+1} < 2$,  $ t_{\ell+1}<r_\ell$ and
\begin{align}
\nonumber &0\le \frac{1}{t_\ell} -\frac{1}{2} \le \frac{1}{n+1}, \hspace{18mm} 0\le \frac{1}{t_{j+1}} -\frac{1}{2} \le \frac{1}{n+1}  ,\\
\label{eq.Afprimaprima} &  0\le \frac{1}{2} -\frac{1}{r_{\ell}} \le \frac{1}{n+1} ,  \hspace{18mm} 0\le \frac{1}{2} +\frac{1}{r_{\ell}} - \frac{1}{t_{\ell +1}} \le  \frac{\beta}{n}.
\end{align}
 (We are applying the results of \cite{R} always with integrability exponents $p,p_\ell=2$.) The previous inequalities imply that 
 \[ 
 0 < \frac{1}{t_{\ell +1}} -\frac{1}{r_{\ell}}  \le \frac{2}{n+1},
 \]
 and this together with the last condition in (\ref{eq.Afprimaprima}) gives the restriction $\beta \ge \max(0,m)$ where $m$ was defined in (\ref{eq.m}).
 
  Now, we can always choose $t_1 =2$, and $t_{\ell +1}$, $r_\ell$, $\ell= 1, \dots, j-1$ such that they satisfy  all the previous conditions and 
\[ \frac{1}{t_{\ell +1}} -\frac{1}{r_{\ell}} = \max{\left( \varepsilon,\frac{1}{2} - \frac{\beta}{n} \right )},\]
 for any $\varepsilon>0$ small. This choice is slightly different from the one in \cite{R}, and it is the only change necessary to extend their results to the range $\beta \ge n/2$.
 
  On the other hand (\ref{eq.Af}) yields
\begin{equation} \label{eq.Af2}
\norm{\widetilde Q_{\theta,j}(q)}_{W^{\alpha,2}}^2 \le  C^{2j} \norm{q}_{W^{\beta,2}}^{2j} C_0^{2\alpha -2\gamma + n} ,
\end{equation}
if we have that $n-1 + 2\alpha - 2\gamma<-1$ which, together with (\ref{eq.Afprima}) and the previous choice of parameters,  implies
\[ 
\alpha < \beta - \frac{1}{2} + (j-1)  -  (j-1) \frac{(n-1)}{2}  \max{\left( \varepsilon,\frac{1}{2} - \frac{\beta}{n} \right )} .\]
 Since we can take $\epsilon>0$ as small as necessary, this gives the condition  $\alpha<\alpha_j$ where $\alpha_j$ was defined in (\ref{eq.alphaj}). 

To show the convergence of the tail  of the series, take any $\alpha \in \RR$. Then, since $\alpha_j$ grows linearly with $j$ in the range $\max(0,m) \le \beta <\infty$, we can find an $l\ge 2$ and  an $\varepsilon = \varepsilon(n,\beta)>0$ such that $\alpha_j-\alpha >\varepsilon j$  if $j\ge l$. Then by (\ref{eq.Af2}) we have
\begin{equation*} 
\norm{ \sum_{j=l}^\infty \widetilde Q_{\theta,j}(q)}_{W^{\alpha,2}}  \le  \sum_{j=l}^\infty \norm{  \widetilde Q_{\theta,j}(q)}_{W^{\alpha,2}} \le  \sum_{j=l}^\infty   C^j \norm{q}_{W^{\beta,2}}^j C_0^{-\varepsilon j} ,
\end{equation*}
and therefore, taking $C_0$ large as specified in the statement of the proposition,  we get the desired result.
\end{proof}
 The following proposition gives an upper bound for the constant in the trace theorem on spheres. 
\begin{proposition} \label{prop:traceT}
Let $f \in W^{1,2}(\RR^n)$, and let $\SP_\rho\subset \RR^n$ be the any sphere of radius $\rho$. Then we have that 
\[\int_{\SP_\rho} |f(x)|^2 \, d\sigma_\rho(x) \le  \int_{\RR^n} |f(x)|^2 \, dx + \int_{\RR^n} |\nabla f(x)|^2 \, dx .\] 
\end{proposition}
\begin{proof}
Assume that $\SP_\rho$ is centered in the origin. The general case follows by the invariance under translations of the Sobolev norm. Without loss of generality consider a  real function $f\in C^\infty_c(\RR^n)$. Then using spherical coordinates with $\theta \in \SP^{n-1}$ we have that
\[\frac{d\, \,}{dr} \left( f(r\theta)^2 r^{n-1}\right) = 2\frac{df}{dr}(r\theta) f(r\theta) r^{n-1} + (n-1)f^2(r\theta)r^{n-2}. \]
Fix $\rho\in(0,\infty)$. If we integrate the previous equation in the $r$ variable, by the fundamental theorem of calculus and the compactness of the support of $f$   we have
\begin{align*}
f(\rho\theta)^2 \rho^{n-1} & = -\int^\infty_\rho 2\frac{df}{dr}(r\theta) f(r\theta) r^{n-1} \, dr - (n-1) \int^\infty_\rho  f^2(r\theta)r^{n-2} \,dr \\
&\le 2\int^\infty_0 |\nabla f(r\theta)| |f(r\theta)| r^{n-1} \, dr,
\end{align*}
since the second integral in the first line is negative.
Then, integrating both sides in the unit sphere $\SP^{n-1}$ we recover the statement of the proposition,
\begin{align*}
\int_{\SP^{n-1}} f(\rho\theta)^2 \rho^{n-1} d\sigma(\theta) &\le 2\int_{\RR^n}|\nabla f(x)| |f(x)|  \, dx  \\
&\le \int_{\RR^n} |f(x)|^2  \, dx + \int_{\RR^n}|\nabla f(x)|^2   \, dx. 
\end{align*}
\end{proof}

We now give the proof of Lemma \ref{lemma.integrals}, used in the estimate of the spherical operator.

\begin{proof}[Proof of Lemma $\ref{lemma.integrals}$]
We can consider only the case of spheres $\SP_\rho$ centered on the origin. By homogeneity, if $y=\rho \theta$ we have
 \[ \int_{\SP_\rho} \frac{1}{|x-y|^{(n-1)-2\gamma}} \, d\sigma_\rho(y) =  \rho^{2\gamma} \int_{\SP^{n-1}} \frac{1}{|x'-\theta|^{(n-1)-2\gamma}} \, d\sigma(\theta),\]
 where $x'\rho= x$ and $\SP^{n-1}$ is the sphere of radius 1 centered on the origin. Hence we need to bound uniformly on $x'$ the last integral. Now, assume that $x'\neq 0$ and take $\omega \in \SP^{n-1}$ such that $\omega = x'/|x'|$. Let $P_\omega = \{x\in \RR^n: x \cdot \omega=0\}$, and let $P(z):= z-(z\cdot\omega)\omega$, be the projection of $z\in \RR^n$ on the plane $P_\omega$. Consider the half sphere comprised between this plane and the parallel one that goes trough $\omega$. The Jacobian of the projection $P$ restricted to $\SP^{n-1}$ is bounded if we exclude a small band of $\varepsilon$ width from it. Let's denote this region by $S_\varepsilon$ (the half sphere minus the band). In the first place we have
 \[\int_{\SP^{n-1}} \frac{1}{|x'-\theta|^{(n-1)-2\gamma}} \, d\sigma(\theta) \le 2n \int_{S_\varepsilon} \frac{1}{|x'-\theta|^{(n-1)-2\gamma}} \, d\sigma(\theta).\]
 This is because in the region $S_\varepsilon$ the integrand has larger values than in the rest of the sphere, since we are in the half which is closer to $x'$, and it is possible to cover generously $\SP^{n-1}$ with $2n$ pieces like $S_\varepsilon$. But since the Jacobian of $P$ is bounded we can use the change of variables $y=P(\theta)$ to integrate in the corresponding region of the plane. Hence
\begin{align*}
\int_{S_\varepsilon} \frac{1}{|x'-\theta|^{(n-1)-2\gamma}}& \, d\sigma(\theta) \le  \int_{S_\varepsilon} \frac{1}{|P(\theta)|^{(n-1)-2\gamma}} \, d\sigma(\theta)\\
&\le C \int_{P(S_\varepsilon)} \frac{1}{|y|^{(n-1)-2\gamma}} \, dy \le C \int_{\RR^{n-1}\cap B_1} \frac{1}{|y|^{(n-1)-2\gamma}} \, dy,
\end{align*}
where we have used that $P(x')=0$. The last integral is finite and it does not depend in any way on $x'$. Therefore we have finished the proof.
\end{proof}

\section*{Acknowledgments}

I am very grateful to  my PhD advisors Alberto Ruiz  and Juan Antonio Barceló for their invaluable advice during the development of this work.  

The author was supported by Spanish government predoctoral grant BES-2015-074055 (project MTM2014-57769-C3-1-P).

  \bibliographystyle{myplainurl}
  \bibliography{references}

\begin{thebibliography}{10}

\bibitem{BFRV10}
J.~A. Barcel\'o, D. Faraco, A. Ruiz, and A. Vargas.
\newblock Reconstruction of singularities from full scattering data by new
  estimates of bilinear {F}ourier multipliers.
\newblock {\em Math. Ann.}, 346(3):505--544, 2010.
\newblock \href {http://dx.doi.org/10.1007/s00208-009-0398-5}
  {\path{doi:10.1007/s00208-009-0398-5}}.

\bibitem{BFRV13}
J.~A. Barcel\'o, D. Faraco, A. Ruiz, and A. Vargas.
\newblock Reconstruction of discontinuities from backscattering data in two
  dimensions.
\newblock {\em SIAM J. Math. Anal.}, 45(6):3494--3513, 2013.
\newblock \href {http://dx.doi.org/10.1137/120902963}
  {\path{doi:10.1137/120902963}}.

\bibitem{BFPRM1}
J.~A. Barcel\'o, M. Folch-Gabayet, S. P\'erez-Esteva, A. Ruiz, and M.~C.
  Vilela.
\newblock A {B}orn approximation for live loads in {N}avier elasticity.
\newblock {\em SIAM J. Math. Anal.}, 44(4):2824--2846, 2012.
\newblock \href {http://dx.doi.org/10.1137/110856265}
  {\path{doi:10.1137/110856265}}.

\bibitem{BLM}
A. Bayliss, Y.~Y. Li, and C.~S. Morawetz.
\newblock Scattering by a potential using hyperbolic methods.
\newblock {\em Math. Comp.}, 52(186):321--338, 1989.
\newblock \href {http://dx.doi.org/10.2307/2008470}
  {\path{doi:10.2307/2008470}}.

\bibitem{BM09quad}
I. Belti\c~t\u a and A. Melin.
\newblock Analysis of the quadratic term in the backscattering transformation.
\newblock {\em Math. Scand.}, 105(2):218--234, 2009.
\newblock \href {http://dx.doi.org/10.7146/math.scand.a-15116}
  {\path{doi:10.7146/math.scand.a-15116}}.

\bibitem{BM09}
I. Belti\c~t\u a and A. Melin.
\newblock Local smoothing for the backscattering transform.
\newblock {\em Comm. Partial Differential Equations}, 34(1-3):233--256, 2009.
\newblock \href {http://dx.doi.org/10.1080/03605300902812384}
  {\path{doi:10.1080/03605300902812384}}.

\bibitem{eskin}
G. Eskin.
\newblock {\em Lectures on linear partial differential equations}, volume 123
  of {\em Graduate Studies in Mathematics}.
\newblock American Mathematical Society, 2011.
\newblock \href {http://dx.doi.org/10.1090/gsm/123}
  {\path{doi:10.1090/gsm/123}}.

\bibitem{Esw}
S. Eswarathasan.
\newblock Microlocal analysis of scattering data for nested conormal
  potentials.
\newblock {\em J. Funct. Anal.}, 262(5):2100--2141, 2012.
\newblock \href {http://dx.doi.org/10.1016/j.jfa.2011.12.013}
  {\path{doi:10.1016/j.jfa.2011.12.013}}.

\bibitem{GS}
I.~M. Gel'fand and G.~E. Shilov.
\newblock {\em Generalized functions. {V}ol. {I}: {P}roperties and operations}.
\newblock Translated by Eugene Saletan. Academic Press, New York-London, 1964.

\bibitem{GU}
A. Greenleaf and G. Uhlmann.
\newblock Recovering singularities of a potential from singularities of
  scattering data.
\newblock {\em Comm. Math. Phys.}, 157(3):549--572, 1993.
\newblock URL: \url{http://projecteuclid.org/euclid.cmp/1104254021}.

\bibitem{OPS}
P. Ola, L. P\"aiv\"arinta, and V. Serov.
\newblock Recovering singularities from backscattering in two dimensions.
\newblock {\em Comm. Partial Differential Equations}, 26(3-4):697--715, 2001.
\newblock \href {http://dx.doi.org/10.1081/PDE-100001768}
  {\path{doi:10.1081/PDE-100001768}}.

\bibitem{PSS}
L. P\"aiv\"arinta, V.~S. Serov, and E. Somersalo.
\newblock Reconstruction of singularities of a scattering potential in two
  dimensions.
\newblock {\em Adv. in Appl. Math.}, 15(1):97--113, 1994.
\newblock \href {http://dx.doi.org/10.1006/aama.1994.1003}
  {\path{doi:10.1006/aama.1994.1003}}.

\bibitem{PSe}
L. P\"aiv\"arinta and V. Serov.
\newblock Recovery of jumps and singularities in the multidimensional
  {S}chr\"odinger operator from limited data.
\newblock {\em Inverse Probl. Imaging}, 1(3):525--535, 2007.
\newblock \href {http://dx.doi.org/10.3934/ipi.2007.1.525}
  {\path{doi:10.3934/ipi.2007.1.525}}.

\bibitem{PSo}
L. P\"aiv\"arinta and E. Somersalo.
\newblock Inversion of discontinuities for the {S}chr\"odinger equation in
  three dimensions.
\newblock {\em SIAM J. Math. Anal.}, 22(2):480--499, 1991.
\newblock \href {http://dx.doi.org/10.1137/0522031}
  {\path{doi:10.1137/0522031}}.

\bibitem{Re}
J.~M. Reyes.
\newblock Inverse backscattering for the {S}chr\"odinger equation in 2{D}.
\newblock {\em Inverse Problems}, 23(2):625--643, 2007.
\newblock \href {http://dx.doi.org/10.1088/0266-5611/23/2/010}
  {\path{doi:10.1088/0266-5611/23/2/010}}.

\bibitem{RRe}
J.~M. Reyes and A. Ruiz.
\newblock Reconstruction of the singularities of a potential from
  backscattering data in 2{D} and 3{D}.
\newblock {\em Inverse Probl. Imaging}, 6(2):321--355, 2012.
\newblock \href {http://dx.doi.org/10.3934/ipi.2012.6.321}
  {\path{doi:10.3934/ipi.2012.6.321}}.

\bibitem{R}
A. Ruiz.
\newblock Recovery of the singularities of a potential from fixed angle
  scattering data.
\newblock {\em Comm. Partial Differential Equations}, 26(9-10):1721--1738,
  2001.
\newblock \href {http://dx.doi.org/10.1081/PDE-100107457}
  {\path{doi:10.1081/PDE-100107457}}.

\bibitem{notasR}
A. Ruiz.
\newblock Harmonic analysis and inverse problems. notes of the 4th summer
  school in inverse problems, oulu, finland, 2002.
\newblock URL:
  \url{http://www.uam.es/gruposinv/inversos/publicaciones/Inverseproblems.pdf}.

\bibitem{RV}
A. Ruiz and A. Vargas.
\newblock Partial recovery of a potential from backscattering data.
\newblock {\em Comm. Partial Differential Equations}, 30(1-3):67--96, 2005.
\newblock \href {http://dx.doi.org/10.1081/PDE-200044450}
  {\path{doi:10.1081/PDE-200044450}}.

\bibitem{ser}
V. Serov.
\newblock Inverse fixed angle scattering and backscattering problems in two
  dimensions.
\newblock {\em Inverse Problems}, 24(6):065002, 14, 2008.
\newblock \href {http://dx.doi.org/10.1088/0266-5611/24/6/065002}
  {\path{doi:10.1088/0266-5611/24/6/065002}}.

\bibitem{stefanov}
P. Stefanov.
\newblock Generic uniqueness for two inverse problems in potential scattering.
\newblock {\em Comm. Partial Differential Equations}, 17(1-2):55--68, 1992.
\newblock \href {http://dx.doi.org/10.1080/03605309208820834}
  {\path{doi:10.1080/03605309208820834}}.

\bibitem{stein}
E.~M. Stein.
\newblock {\em Singular integrals and differentiability properties of
  functions}.
\newblock Princeton Mathematical Series, No. 30. Princeton University Press,
  Princeton, N.J., 1970.

\end{thebibliography}

\end{document}